\documentclass[11pt]{article}
\usepackage[utf8]{inputenc}
\usepackage{bbm}
\usepackage{bigints}
\usepackage[a4paper]{geometry}
\usepackage{amsthm}
\usepackage{amssymb}
\usepackage{color}
\usepackage{tikz}
\usepackage{enumerate}

\newtheorem{theorem}{Theorem}
\newtheorem{proposition}{Proposition}
\newtheorem{lemma}{Lemma}

\newtheorem{assumption}{Assumption}
\newtheorem*{aprime}{Assumption (A')}
\newtheorem*{hhom}{Assumption ($\text{H}_{\text{hom}}$)}
\newtheorem*{hinfty}{Assumption ($\text{H}_{\infty}$)}
\newtheorem*{hinftyprime}{Assumption ($\text{H'}_{\infty}$)}
\theoremstyle{definition}
\newtheorem{definition}{Definition}

\theoremstyle{remark}
\newtheorem{remark}{Remark}

\def\E{\mathbb{E}}
\def\P{\mathbb{P}}
\def\R{\mathbb{R}}
\def\Q{\mathbb{Q}}
\def\N{\mathbb{N}}
\def\1{\mathbbm{1}}
\def\d{\partial}
\def\Z{\mathbb{Z}}

\def\cF{{\cal F}}

\def\cB{{\cal B}}
\def\cM{{\cal M}}
\def\cC{{\cal C}}

\begin{document}

\title{$Q$-processes and asymptotic properties of Markov processes conditioned not to hit moving boundaries}
\author{William Oçafrain$^{1}$}
\date{\today}

\footnotetext[1]{Institut de Math\'{e}matiques de Toulouse, UMR 5219; Universit\'{e} de Toulouse, CNRS, UPS IMT, F-31062
Toulouse Cedex 9, France; \\
  E-mail: william.ocafrain@math.univ-toulouse.fr}

\maketitle

\begin{abstract}
We investigate some asymptotic properties of general Markov processes conditioned not to be absorbed by the moving boundaries. We first give general criteria involving an exponential convergence towards the $Q$-process, that is the law of the considered Markov process conditioned never to reach the moving boundaries. This exponential convergence allows us to state the existence and uniqueness of the quasi-ergodic distribution considering either boundaries moving periodically or stabilizing boundaries. We also state the existence and uniqueness of a quasi-limiting distribution when absorbing boundaries stabilize. We finally deal with some examples such as diffusions which are coming down from infinity.
\end{abstract} 

\bigskip

\bigskip

\tableofcontents

\section{Introduction}
\label{introduction}
Let $(\Omega, {\cal A}, \P)$ be a probability space and let $(X_t)_{t \in I}$ be a time-homogeneous Markov process (where $I=\Z_+$ or $\R_+$) defined on a metric state space $(E,d)$. We associate with $E$ a $\sigma$-algebra ${\cal E}$. For any $t \in I$, denote by ${\cal F}_{t} = \sigma(X_s, 0 \leq s \leq t)$ the $\sigma$-field generated by $(X_s)_{0 \leq s \leq t \in I}$.  For any subset $F \subset E$, denote by ${\cal M}_1(F)$ the set of probability measures defined on $F$ and ${\cal B}(F)$ the set of the bounded measurable function $f \colon F \to \R$. 

We define, for each time $t \in I$, a subset $A_t \in {\cal E}$ called \textit{absorbing subset at time $t$} and we denote by $E_t$ the complement set of $A_t$ called \textit{survival subset at time $t$}. We will call $t \mapsto A_t$ \textit{the moving absorbing subset} or \textit{the moving absorbing boundary}. 
We denote by 
\begin{equation*}
\label{tau}
    \tau_A := \inf\{t \in I : X_t \in A_t\}
\end{equation*}
the reaching time of  $(A_t)_{t \in I}$ by the process $(X_t)_{t \in I}$. In all what follows, we will assume that $\tau_{A}$ is a stopping time for the filtration $(\cF_t)_{t \in I}$. This assumption holds when, for example, the Markov process $(X_t)_{t \in I}$ is continuous and all the sets $(A_t)_{t \in I}$ are closed. 

Even though the process $(X_t)_{t \in I}$ is time-homogeneous, we will associate to this process a family of probability measures $(\P_{s,x})_{s \in I, x \in E}$ such that, for any $s \in I$ and for any $x \in E$, $\P_{s,x}(X_s = x)=1$ and, for any measure $\mu$ on $E$, define $\P_{s,\mu} = \int \P_{s,x} d\mu(x)$. We denote by $\E_{s,x}$ and $\E_{s,\mu}$ the corresponding expectations. When the starting time is not needed, we will prefer the notation $\P_\mu := \P_{0,\mu}$ and $\E_\mu := \E_{0,\mu}$. 

In this paper, we will deal with the so-called \textit{$Q$-process}, \textit{quasi-limiting distribution} and \textit{quasi-ergodic distribution}, defined as below :
\begin{definition}
\label{definition}
\begin{enumerate}[i)]
\item We say that there is a \textit{$Q$-process} if there exists a family of probability measures $(\Q_{s,x})_{s \in I, x \in E_s}$ such that for any $s \leq t$, $x \in E_s$,
$$\P_{s,x}(X_{[s,t]} \in \cdot | \tau_{A} > T)  \underset{T \in I, T \to \infty}{\overset{(d)}{\longrightarrow}} \Q_{s,x}(X_{[s,t]} \in \cdot), $$
where, for any $u,v \in I$, $X_{[u,v]}$ is the trajectory of $(X_t)_{t \in I}$ between times $u$ and $v$ and where $(d)$ refers to the weak convergence of probability measures.
\item We say that $\alpha \in {\cal M}_1(E)$ is a \textit{quasi-limiting distribution} if, for some $\mu \in {\cal M}_1(E_0)$, 
\begin{equation}
\label{conve}
\P_{\mu}(X_t \in \cdot | \tau_A > t) \underset{t \in I, t \to \infty}{\overset{(d)}{\longrightarrow}} \alpha.\end{equation}
\item We say that $\beta \in {\cal M}_1(E)$ is a \textit{quasi-ergodic distribution} if there exists $\mu \in {\cal M}_1(E_0)$ such that,
\begin{itemize}
\item 
 $$\frac{1}{n} \sum_{k=0}^n \P_{\mu}(X_k \in \cdot | \tau_A > n) \underset{n \to \infty}{\overset{(d)}{\longrightarrow}} \beta$$
if $I = \Z_+$,
\item
$$\frac{1}{t} \int_0^t \P_{\mu}(X_s \in \cdot | \tau_A > t)ds \underset{t \to \infty}{\overset{(d)}{\longrightarrow}} \beta$$
if $I = \R_+$.
\end{itemize}
\end{enumerate} 
\end{definition}
For Markov processes absorbed by non-moving boundaries (i.e. $A_t = A_0$ for any $t \in I$), the notions of $Q$-process, quasi-limiting distribution and quasi-ergodic distribution are dealt with by the theory of \textit{quasi-stationarity}, which studies the asymptotic behavior of such processes conditioned not to be absorbed. In particular, the main object of this theory is the \textit{quasi-stationary distribution}, which is defined as a probability measure $\alpha$ such that, for all $t \in I$,
\begin{equation}
\label{qsd}
\P_{\alpha}(X_t \in \cdot | \tau_A > t) = \alpha.
\end{equation}   
In the time-homogeneous setting, it is well known that the notions of quasi-stationary distributions and quasi-limiting distributions are equivalent. 
The interested reader can see \cite{MV2012} and \cite{CMSM} for an overview of the theory. In particular, these monographes give some results about the existence of quasi-limiting distributions and $Q$-processes for several processes : Markov chains on finite state space and countable space, birth and death processes, diffusion processes and others. In a same way, existence of quasi-ergodic distributions has been also shown for such processes. The reader can see \cite{GHY2016, ZLS2013, BR1999} for the study on quasi-ergodic distributions in a very general framework.  

In this article, we will be interested in the existence of a $Q$-process, a quasi-limiting distribution and a quasi-ergodic distribution when $(A_t)_{t \in I}$ depends on the time. More precisely, we want to generalize the results presented in \cite{OCAFRAIN2017}, which were only obtained for discrete-time Markov chains defined on finite state space. In particular, this paper showed, in a first time, that the notion of quasi-stationary distribution as defined by the relation \eqref{qsd}, considering that the boundary $(A_n)_{n \in \Z_+}$ is moving, is not well-defined. If moreover the boundary moves periodically, then the notion of quasi-limiting distribution is not well-defined either. Finally, it is shown in \cite{OCAFRAIN2017} that, still considering periodic moving boudaries, the probability measure 
$$\frac{1}{n} \sum_{i=1}^n \P_\mu(X_k \in \cdot | \tau_A > n)$$
converges weakly towards a quasi-ergodic distribution $\beta$ if the initial measure $\mu$ satisfies some assumptions (see \cite[Theorem 3] {OCAFRAIN2017}). Moreover, the $Q$-process is well-defined.

Hence, the main goal of this paper is to recover these results for a more wide class of Markov processes, such as diffusion processes. In particular, we want to know if the quasi-ergodic distribution is still well-defined for such processes when the moving boundary $(A_t)_{t \in I}$ is periodic. 
 
The main assumption that $(X_t)_{t \in I}$ will satisfy in this paper will be based on a Champagnat-Villemonais type condition. When $A$ does not depend on $t$, Champagnat and Villemonais introduce in \cite{CV2014} the following assumption : there exists $\nu \in {\cal M}_1(E)$ such that
\begin{enumerate}[({A}1)]
\item there exist $t_0 \geq 0$ and $c_1 > 0$ such that
$$\forall x \in E_0, ~~\P_{x}(X_{t_0} \in \cdot | \tau_A > t_0) \geq c_1 \nu;$$
\label{CV1}
\item there exists $c_2 > 0$ such that : $\forall x \in E_0$, $\forall t \geq 0$,
$$\P_{\nu}(\tau_A > t) \geq c_2 \P_{x}(\tau_A > t).$$
\label{CV2} 
\end{enumerate}
In particular, (A1) can be seen as a conditional version of Doeblin's condition. Then the authors show that (A1)-(A2) are equivalent to an exponential uniform convergence of the total variation distance between the conditional probability $\P_\mu(X_t \in \cdot | \tau_A > t)$ and the unique quasi-stationary distribution. Moreover, one has, under these assumptions, the existence of a $Q$-process, as well as the existence and the uniqueness of the quasi-ergodic distribution (see \cite{CV2017} for this last result).

Champagnat and Villemonais also adapt the assumptions (A1)-(A2) to the time-inhomogeneous setting in the paper \cite{CV2016}. This time-inhomogeneous version will be used to our purpose; we refer the reader to the Section 3 for more details about it. In particular, the Assumption (A'), which is introduced in Section 2, is a particular case of their time-inhomogeneous conditions. In this paper, the existence of a $Q$-process will be proved, as well as the exponential convergence in total variation of the probability measure $\P_{s,x}( X_{[s,t]} \in \cdot | \tau_A > T)$ towards the $Q$-process, when $T$ goes to infinity. In the same way as in the paper \cite{CV2017}, this exponential convergence implies that the existence and the uniqueness of the quasi-ergodic distribution is equivalent to an ergodic theorem for the $Q$-process. In particular, this corollary will be applied for periodic moving boundaries to show the existence and the uniqueness of a quasi-ergodic distribution. 

Moreover, the case of a non-increasing converging moving boundary (the notion of convergence will be defined further) will be dealt with. In this case, one can expect an asymptotic homogeneity of the conditional probability $\P_{s,x}(X_{s+t} \in \cdot | \tau_A > s+t)$ when $s$ goes to infinity (in the meaning of Proposition \ref{uc} in Subsection \ref{subsection-conv}), and use this property to show the existence of a quasi-limiting distribution. It will be therefore shown in this paper that, under the Champagnat-Villemonais condition and some extra assumptions, there exists a unique quasi-limiting distribution for which the weak convergence \eqref{conve} holds for any initial law $\mu$. 

This paper ends with an application of these results to a one-dimensional diffusion process \textit{coming down from infinity}, that is to say, for some $t \geq 0$ and $y \in \R_+$,
$$\lim_{x \to +\infty} \P_x(\tau_y < t) > 0,$$
where $\tau_y$ is the hitting time of $y$ by $(X_t)_{t \in I}$. It will be shown that, under additional assumptions, the diffusion process $(X_t)_{t \geq 0}$ satisfies the time-inhomogeneous Champagnat-Villemonais conditions. 

\section{Assumptions and general results}
\label{general-results}
From now on, assume that $(A_t)_{t \in I}$ could depend on time and for any $s \in I$ and $x \in E_s$,
$$\P_{s,x}(\tau_{A} < \infty) = 1,$$
and, in order to make sense of the conditioning, we will assume that for any $s \leq t$ and any $x \in E_s$, 
$$\P_{s,x}(\tau_{A} > t) > 0.$$ 
We introduce now the main assumption adapted from the Champagnat-Villemonais conditions introduced in \cite{CV2014}:
\begin{aprime}
\label{assumption}
There exist $(\nu_s)_{s \in I}$ a sequence of probability measures $(\nu_s \in {\cal M}_1(E_s)$ for each $s \in I$), and $t_0,c_1,c_2 > 0$ such that
\begin{enumerate}[({A'}1)]
\item
$\forall s \in I, \forall x \in E_s$, $$\P_{s,x}(X_{s+t_0} \in \cdot | \tau_{A} > s+t_0) \geq c_{1} \nu_{s+t_0};$$
\item $\forall s \leq t, \forall x \in E_s$,
$$\P_{s,\nu_s}(\tau_{A} > t) \geq c_{2} \P_{s,x}(\tau_{A} > t).$$
\end{enumerate}
\end{aprime}
In this section, the main results and contributions in this paper are presented. Let us recall that the \textit{total variation distance} between two probability measures $\mu$ and $\nu$ on $E$ is defined by 
$$||\mu - \nu||_{TV} := \sup_{f \in \cB_1(E)} |\mu(f) - \nu(f)|,$$
where $\cB_1(E) := \{ f \in \cB(E) : ||f||_\infty \leq 1\}$ and where the notation 
$$\mu(f) := \int_E f(x) \mu(dx)$$
is used. Then let us state our main result :
\begin{theorem}
\label{theoreme-1}
Under Assumption (A'), there exists a $Q$-process (Definition \ref{definition} (i)). Furthermore, there exists $C,\lambda > 0$ such that, for any $s \leq t \leq T$ and $x \in E_s$,
\begin{align*}||\P_{s,x}(X_{[s,t]} \in \cdot | \tau_{A} > T) - \Q_{s,x}(X_{[s,t]} \in \cdot) ||_{TV} \leq  C e^{-\lambda (T-t)}.
\end{align*}
\end{theorem}
Explicit formulae will be provided later in Theorem \ref{expo-conv-thm}, whose the statement is more precise than the one of the previous theorem.
 
As written in the introduction, two specific behavior of moving behavior will be studied in this paper :
\begin{itemize}
\item Periodic moving boundaries,
\item Non-increasing converging moving boundaries, i.e. $A_t \subset A_s$ for all $s \leq t$ and 
\begin{equation}
\label{A_inf}
A_\infty := \bigcap_{t \in I} A_t \ne \emptyset.
\end{equation}
\end{itemize}
In the periodic case, the following theorem is shown in the Subsection \ref{subsection-per}.
\begin{theorem}
\label{qed-presentation}
If $(X_t)_{t \geq 0}$ satisfies Assumption (A'), 
then there exists a unique probability measure $\beta$ such that, for any $\mu \in {\cal M}_1(E_0)$,
$$\frac{1}{t} \int_{0}^t \P_\mu(X_s \in \cdot | \tau_A > t)ds \underset{t \to \infty}{\overset{(d)}{\longrightarrow}} \beta.$$
\end{theorem}
The expression of the quasi-ergodic distribution $\beta$ is spelled out later in Theorem \ref{qed}. 

For converging non-increasing moving boundaries, some extra assumptions are needed to state the theorems. The following assumptions will be useful to show the asymptotic homogeneity of the conditional probability $\P_{s,x}(X_{s+t} \in \cdot | \tau_A > s+t)$ :
\begin{hhom}
\label{indis}
\begin{enumerate}[a)]
\item \textbf{Strong Markov property:} For any $\tau$ stopping time of $\cF_t = \sigma(X_s, 0 \leq s \leq t)$ and for any $x \in E$, 
$$\P_x((X_{\tau+t})_{t \in I} \in \cdot, \tau < \infty | \cF_\tau) = \1_{\tau < \infty} \P_{X_\tau}((X_t)_{t \in I} \in \cdot);$$
\item \textbf{Convergence in law for the hitting times :} For any $x \in E_0$ and for any $t \in I$,
$$\P_{s,x}(\tau_A > s+t) \underset{s \to + \infty}{\longrightarrow} \P_x(\tau_{A_\infty} > t),$$ where $\tau_{A_\infty} := \inf\{t \geq 0 : X_t \in A_\infty\}$; 
\item \textbf{Time-continuity:} For any $x \in E_0$ and $s \geq 0$, the functions $t \to \P_{s,x}(\tau_{A} > t)$ and $t \to \P_x(\tau_{A_\infty} > t)$ are continuous;
\end{enumerate}
\end{hhom}
 Moreover, defining $E_\infty$ as the complement of $A_\infty$, let us set the additional following assumption :
  \begin{hinfty}
There exists a unique probability measure $\alpha_\infty \in \cM_1(E_\infty)$ such that, for any $\mu \in \cM_1(E_\infty)$ and $t \geq 0$,
\begin{equation}
\label{gamma}
||\P_\mu(X_t \in \cdot | \tau_{A_\infty} > t) - \alpha_\infty||_{TV} \leq C_\infty e^{-\gamma_\infty t},\end{equation}
where $C_\infty, \gamma_\infty > 0$.  
\end{hinfty}
Under Assumption $(H_\infty)$, it is well known (see \cite{MV2012}) that there exists $\lambda_\infty > 0$ such that, for any $t \in I$,
\begin{equation}
\label{lambda}
\P_{\alpha_\infty}(\tau_\infty > t) = e^{- \lambda_\infty t},\end{equation}
and also a function $\eta_\infty$ (see \cite[Proposition 2.3]{CV2014}) positive on $E_\infty$ and vanishing on $A_\infty$ such that, for any $x \in E_\infty$, 
\begin{equation}
\label{def-eta}
\eta_\infty(x) = \lim_{t \to \infty} e^{\lambda_\infty t} \P_x(\tau_{A_\infty} > t).\end{equation}
To state our result of convergence, the following assumption is needed :
\begin{hinftyprime}
There exists $s_0 \in I$ and $x_0 \in E_{s_0}$ such that, for any $s \geq s_0$, 
$$\E_{s,x_0}\left[e^{\lambda_\infty \tau_A} \eta_\infty(X_{\tau_A})\right] < + \infty,$$
and 
$$\lim_{s \to \infty} \E_{s,x_0}\left[e^{\lambda_\infty (\tau_A-s)} \eta_\infty(X_{\tau_A})\right] = 0.$$
\end{hinftyprime}
Somehow, this previous assumption impose that the boundary $(A_t)_{t \in I}$ decreases fast enough towards $A_\infty$. 

Then, considering non-increasing converging moving boundaries, one has the following statement :
\begin{theorem}
\label{first-thm}
Under the assumptions (A'), $(H_{hom})$ $(H_\infty)$ and $(H'_\infty)$,
for any $\mu \in {\cal M}_1(E_0)$, 
$$\P_\mu(X_t \in \cdot | \tau_A > t) \underset{t \to \infty}{\overset{(d)}{\longrightarrow}} \alpha_\infty,$$
where $\alpha_\infty$ is the quasi-stationary distribution defined in the Assumption ($H_\infty$).
\end{theorem}
The existence and the uniqueness of the quasi-ergodic distribution is also shown in the Subsection \ref{subsection-conv}.


\section{Exponential convergence towards $Q$-process and quasi-ergodic distribution}

First, we recall Proposition 3.1. and Theorem 3.3. of \cite{CV2016}. In their paper, N. Champagnat and D. Villemonais took a time-inhomogeneous Markov process and $(Z_{s,t})_{s \leq t}$ a collection of multiplicative nonnegative random variables (i.e. satisfying $Z_{s,r}Z_{r,t} = Z_{s,t},~~\forall s \leq r \leq t$) such that, for any $s \leq t \in I$ and $x \in E_s$, $\E_{s,x}(Z_{s,t}) > 0$ and $\sup_{y \in E_s} \E_{s,y}(Z_{s,t}) < \infty$. In our case, $(X_t)_{t \in I}$ is time-homogeneous, however the penalization $(Z_{s,t})_{s \leq t}$ we shall use is given by
$$Z_{s,t} = \1_{\tau_{A} > t},~~~~\forall s \leq t.$$
and is time-inhomogeneous because $(A_t)_{t \in I}$ depends on $t$. For any $s \leq t$, define by
$$\phi_{t,s} \colon \mu \mapsto \P_{s,\mu}(X_{t} \in \cdot |\tau_{A} > t).$$
Then, by Markov property, the family $(\phi_{t,s})_{s \leq t}$ is a semi-flow, that is : for any $r \leq s \leq t$,
\begin{equation}
\label{semi-flow-property}
\phi_{t,r} = \phi_{t,s} \circ \phi_{s,r}.
\end{equation}
Let $t_0 \in I$. 
For any $s \geq t_0$ and $x_1,x_2 \in E_{s-t_0}$, define $v_{s,x_1,x_2}$ and $v_{s}$ as follows :
\begin{align}
\label{vessie}
&v_{s,x_1,x_2} = \min_{j=1,2} \phi_{s,s-t_0} (\delta_{x_j});\\
&v_{s} =  \min_{x \in E_{s-t_0}} \phi_{s,s-t_0} (\delta_{x}),
\label{vessie-prime}
\end{align}
where the minimum of several measures is understood as the largest measure smaller than all the considered measures. 
Finally, for any $s \geq t_0$, define
\begin{align}
\label{dsi}
&d_{s} = \inf_{t \geq 0,x_1,x_2 \in E_{s-t_0}} \frac{\P_{s,v_{s,x_1,x_2}}(\tau_{A }> t+s)}{\sup_{x \in E_{s}} \P_{s,x}(\tau_{A} > t+s)};\\
&d'_{s} =  \inf_{t \geq 0} \frac{\P_{s,v_{s}}(\tau_{A} > s+t)}{\sup_{x \in E_{s}} \P_{s,x}(\tau_{A} > s+t)}.
\label{dsiprime}
\end{align}
In particular, $v_s \leq v_{s,x_1,x_2}$ and $d_s' \leq d_s$. 
We can now state Proposition 3.1. and Theorem 3.3. of \cite{CV2016} in our situation (see \cite{CV2016} for a more general framework) : 
\begin{proposition}[Proposition 3.1. (\cite{CV2016})]
\label{proposition-CV}
For any $s \in I$ such that $d'_s > 0$ and $y \in E_s$, there exists a finite constant $C_{s,y}$ only depending on $s$ and $y$ such that, for all $x \in E_s$ and $t,u \geq s+t_0$ with $t \leq u$,
\begin{equation}
\label{4}
\left| \frac{\P_{s,x}(\tau_{A } > t)}{\P_{s,y}(\tau_{A} > t)} - \frac{\P_{s,x}(\tau_{A} > u)}{\P_{s,y}(\tau_{A} > u)} \right| \leq C_{s,y} \inf_{v \in [s+t_0,t]} \frac{1}{d'_v} \prod_{k=0}^{\left \lfloor \frac{v-s}{t_0} \right \rfloor-1} (1-d_{v-k}).
\end{equation}
In particular, if
\begin{equation}
\label{conv}
\liminf_{t \in I, t \to \infty} \frac{1}{d'_t} \prod_{k=0}^{\left \lfloor \frac{t-s}{t_0} \right \rfloor-1} (1-d_{t-k})= 0,
\end{equation}
for all $s \geq 0$, there exists a positive bounded function $\eta_s : E_s \to (0,\infty)$ such that
$$\lim_{t \to \infty}\frac{\P_{s,x}(\tau_{A} > t)}{\P_{s,y}(\tau_{A} > t)} = \frac{\eta_s(x)}{\eta_s(y)},~~~~\forall x,y \in E_s,$$
where, for any fixed $y$, the convergence holds uniformly in $x$. $\eta_s$ satisfies for all $x \in E_s$ and $s \leq t \in I$,
$$\E_{s,x}(\1_{\tau_{A} > t} \eta_t(X_ {t})) = \eta_s(x).$$
In addition, the function $s \to ||\eta_s||_\infty$ is locally bounded on $[0,\infty)$.
\end{proposition} 
\begin{theorem}[Theorem 3.3 (\cite{CV2016})]
\label{theorem-CV}
Assume that 
$$\liminf_{t \in I, t \to \infty}\frac{1}{d'_t} \prod_{k=0}^{\left \lfloor \frac{t-s}{t_0} \right \rfloor-1} (1-d_{t-k}).$$
Then there exists $(\Q_{s,x})_{s \in I, x \in E_s}$ such that
$$\P_{s,x}(X_{[s,s+t]} \in \cdot | \tau_{A} >  T)  \underset{T \in I, T \to \infty}{\overset{(d)}{\longrightarrow}} \Q_{s,x}(X_{[s,s+t]} \in \cdot),~~\forall s,t \in I, x\in E_s, $$
and $\Q_{s,x}$ is given by, for all $s \leq t$ and $x\in E_s$,
\begin{equation}
\label{Q-proc}\Q_{s,x}(X_{[s,t]} \in \cdot) = \E_{s,x}\left( \1_{X_{[s,t]} \in \cdot} \frac{\1_{\tau_{A} > t} \eta_t(X_{t})}{\E_{s,x}(\1_{\tau_{A} > t} \eta_t(X_{t}))}\right) =  \E_{s,x}\left( \1_{X_{[s,t]} \in \cdot, \tau_{A} > t} \frac{\eta_t(X_{t})}{\eta_s(x)}\right).\end{equation}
Furthermore, under $(\Q_{s,x})_{s \in I, x \in E_s}$, $(X_t)_{t \in I}$ is a time-inhomogeneous Markov process. Finally, this process is asymptotically mixing in the sense that, for any $s \leq t$ and for any $\mu, \pi \in \cM_1(E_s)$,
$$||\Q_{s,\mu}(X_t \in \cdot) - \Q_{s,\pi}(X_t \in \cdot)||_{TV} \leq 2  \prod_{k=0}^{\left \lfloor \frac{t-s}{t_0} \right \rfloor-1} (1-d_{t-k}),$$
where 
\begin{equation}
\label{def}
\Q_{s,\mu}(\cdot) := \int_{E_s} \Q_{s,x}(\cdot) \mu(dx).\end{equation}
\end{theorem}
\begin{remark}
Note that, by the definition \eqref{def}, when $\mu$ is not a Dirac mass,
$$\Q_{s,\mu} \ne \lim_{T \to + \infty} \P_{s,\mu}( \cdot | \tau_A > T).$$
However, using the notation
\begin{equation}
\label{notation}
f * \mu(dx) := \frac{f(x) \mu(dx)}{\mu(f)},~~~~\forall \mu \in \cM_1(E_s),~\forall f \in \cB(E_s),\end{equation}
one has
$$ \lim_{T \to + \infty} \P_{s,\mu}( \cdot | \tau_A > T) = \Q_{s,\eta_s * \mu}.$$
\end{remark} 
\begin{remark}
We emphasize that,  in \cite{CV2016}, Proposition 3.1. and Theorem 3.3 are stated for any penalizations $(Z_{s,t})_{s \leq t}$. In particular, instead of considering absorbed Markov process, it is possible to work on renormalized Feynman-Kac semi-group taking 
$$Z_{s,t} = e^{\int_s^t g(X_u)du},$$
for some measurable functions $g$. Indeed, the specific choice of $Z_{s,t}$ we did in Proposition \ref{proposition-CV} and Theorem \ref{theorem-CV} does not play a role in the proofs.
\end{remark}
Under Assumption (A'), and considering $t_0 \in I$ as defined in Assumption (A'), one has, for any $s \in I$,
\begin{equation}
\label{dsprime}d_s \geq d'_s \geq c_{1}c_{2} > 0.
\end{equation}
Hence, by Proposition 1, \eqref{conv} is satisfied and, for any $s < s+t_0 \leq t \leq u$ and $x,y \in E_s$,
\begin{equation}
\label{maj}
\left| \frac{\P_{s,x}(\tau_{A} > t)}{\P_{s,y}(\tau_{A} > t)} - \frac{\P_{s,x}(\tau_{A} > u)}{\P_{s,y}(\tau_{A} > u)} \right| \leq C_{s,y} \times \frac{1}{c_1c_2} (1-c_1c_2)^{ \left \lfloor \frac{t-s}{t_0} \right \rfloor}. 
\end{equation}
From this last equation, we can expect an exponential convergence of the family of probability measures $(\P_{s,x}( X_{[s,t]} \in \cdot | \tau_{A} > T))_{T \geq t}$ towards the $Q$-process. Let us now reformulate the Theorem \ref{theoreme-1}, in a more precise manner :
\begin{theorem}
\label{expo-conv-thm}
Let $(X_t)_{t \geq 0}$ be a Markov process satisfying Assumption (A'). 
\begin{enumerate}
\item
Then, for any $s \leq t \leq T$ and $x \in E_s$, 
\begin{align*}||\P_{s,x}(X_{[s,t]} \in \cdot | \tau_{A} > T) - \Q_{s,x}(X_{[s,t]} \in \cdot) ||_{TV} \leq   \frac{1}{(c_1c_2)^3} (1-c_1c_2)^{\left\lfloor \frac{T-t}{t_{0}} \right\rfloor},
\end{align*}
where $\Q_{s,x}$ is defined by \eqref{Q-proc} in Theorem \ref{theorem-CV}.
\item If the $Q$-process satisfies an ergodic theorem, i.e. there exists a probability measure $\beta$ such that for any $x \in E_0$,
\begin{equation}
\label{beta}
\frac{1}{t} \int_0^t \Q_{0,x}(X_s \in \cdot)ds \underset{t \to \infty}{\overset{(d)}{\longrightarrow}} \beta,
\end{equation}
then for any $\mu \in \cM_1(E_0)$,
$$\frac{1}{t} \int_0^t \P_{\mu}(X_s \in \cdot|\tau_A > t) ds \underset{t \to \infty}{\overset{(d)}{\longrightarrow}} \beta.$$
\end{enumerate}
\end{theorem}
The statement of this theorem is implicitly written for $I = \R_+$. Obviously, the statement holds when $I = \Z_+$ and, from now, we will confuse integral and sum to deal with quasi-ergodic distributions when the time space $I$ will not be specify. 

\begin{proof}[Proof of Theorem \ref{expo-conv-thm}]
First we will show the exponential convergence towards the $Q$-process essentially thanks to \eqref{maj}. In the second step, we will show the existence and uniqueness of the quasi-ergodic distribution using a method similar to that used in \cite{CV2017}. 

\begin{enumerate}[{Step} 1]
\item : Exponential convergence towards the $Q$-process

We may extend \eqref{maj} to general initial law $\mu$ and $\pi$ : putting moreover $1/c_1c_2$ inside the constant, there exists $C_{s,\pi} > 0$ only depending on $s$ and $\pi$ such that, for any $s \leq t \leq u$,
$$\left| \frac{\P_{s,\mu}(\tau_{A} > u)}{\P_{s,\pi}(\tau_{A} > u)} - \frac{\P_{s,\mu}(\tau_{A} > t)}{\P_{s,\pi}(\tau_{A} > t)}\right| \leq C_{s,\pi} (1-c_1c_2)^{ \left\lfloor \frac{t-s}{t_{0}} \right\rfloor }.$$
Thus, by Theorem \ref{theorem-CV} and letting $u \to \infty$,
\begin{equation}
\label{expo-conv}
\left| \frac{\mu(\eta_s)}{\pi(\eta_s)} - \frac{\P_{s,\mu}(\tau_{A} > t)}{\P_{s,\pi}(\tau_{A} > t)}\right| \leq C_{s,\pi}(1-c_1c_2)^{ \left\lfloor \frac{t-s}{t_{0}} \right\rfloor }.\end{equation}
Using Markov property,  for any $s \leq t \leq T$ and for any $x \in E_s$,
\begin{align*}
\P_{s,x}(X_{[s,t]} \in \cdot | \tau_{A} > T) &= \E_{s,x}\left( \1_{X_{[s,t]} \in \cdot} \frac{\1_{\tau_{A} > t}\P_{t,X_{t}}(\tau_{A} > T)}{\P_{s,x}(\tau_{A} > T)}\right) \\
&= \E_{s,x}\left( \1_{X_{[s,t]} \in \cdot} \frac{\1_{\tau_{A} > t}\P_{t,X_{t}}(\tau_{A} >T)}{\E_{s,x}(\1_{\tau_{A} >t} \P_{t,X_{t}}(\tau_{A} >T))}\right) \\
&= \E_{s,x}\left( \1_{X_{[s,t]} \in \cdot}\1_{\tau_{A} > t} \frac{\P_{t,X_{t}}(\tau_{A} > T)}{\P_{s,x}(\tau_{A } >t) \E_{s,x}(\P_{t,X_{t}}(\tau_{A } > T) | \tau_A > t)}\right) \\
&= \E_{s,x}\left( \1_{X_{[s,t]} \in \cdot}\1_{\tau_{A} > t} \frac{\P_{t,X_{t}}(\tau_{A} > T)}{\P_{s,x}(\tau_{A } >t) \P_{t,\phi_{t,s}(\delta_x)}(\tau_{A } > T))}\right).
\end{align*} 
Using this last equality and \eqref{Q-proc}, for any $s \leq t \leq T$, for any $x \in E_s$ and any $B \in {\cal E}$,
\begin{multline*}
\left| \P_{s,x}(X_{[s,t]} \in B | \tau_{A} > T)- \Q_{s,x}(X_{[s,t]} \in B) \right| \\ = \left| \E_{s,x}\left( \frac{\1_{X_{[s,t]} \in B}\1_{\tau_{A } > t}}{\P_{s,x}(\tau_{A } > t) } \left(\frac{\P_{t,X_{t}}(\tau_{A } > T)}{\P_{t,\phi_{t,s}(\delta_x)}(\tau_{A} >T))} - \frac{\eta_{t}(X_{t})}{\phi_{t,s}(\delta_x)(\eta_{t})}\right)\right) \right| \\
\leq C_{t,\phi_{t,s}(\delta_x)}(1-c_1c_2)^{ \left\lfloor \frac{T-t}{t_{0}} \right\rfloor } \E_{s,x}\left( \frac{\1_{X_{[s,t]} \in B}\1_{\tau_{A } >t}}{\P_{s,x}(\tau_{A } >t) }\right),
\end{multline*}
where the last inequality follows from \eqref{expo-conv}. Moreover, for any $s \leq t$, 
$$\E_{s,x}\left( \frac{\1_{X_{[s,t]} \in B}\1_{\tau_{A } >t}}{\P_x(\tau_{A } > t) }\right) = \P_{s,x}\left({X_{[s,t]} \in B}|{\tau_{A } >t}\right) \leq 1,~~~~\forall B \in {\cal E}.$$
Hence, for any $s \leq t$, $x \in E_s$ and $B \in {\cal E}$,
$$\left| \P_{s,x}(X_{[s,t]} \in B | \tau_{A} >T)- \Q_{s,x}(X_{[s,t]} \in B) \right| \leq  C_{t,\phi_{t,s}(\delta_x)}(1-c_1c_2)^{ \left\lfloor \frac{T-t}{t_{0}} \right\rfloor }.$$
Without loss of generality, one can assume $t-s \geq t_0$, since for any $t \leq s + t_0$,
$$\{X_{[s,t]} \in B\} = \{X_{[s,s+t_0]} \in \tilde{B}\},$$
where $\tilde{B} := \{\omega : [s,s+t_0] \to E : \omega_{[s,t]} \in B\}$ is a measurable set.  \\
Note that \cite{CV2016} provides an explicit formula of $C_{s,y}$ in the proof of Proposition 3.1. for $s$ and $y$ fixed. Adapting this formula for a general probability measure $\pi$ and recalling that we put the term $1/c_1c_2$ inside $C_{s,\pi}$, one explicit formula of $C_{s,\pi}$ for $s \in I$ can be
\begin{equation}
\label{cont}
C_{s,\pi} = \frac{1}{c_1c_2} \frac{\sup_{z \in E_s} \P_{s,z}(\tau_{A } > v_s)}{d'_{v_s} \P_{s,\pi}(\tau_{A} > v_s )},\end{equation}
where $v_s \in I$ is the smaller time $v \geq s+t_0$ such that $d'_v > 0$ (with $d'_v$ as defined in \eqref{dsiprime}). Then, by \eqref{dsprime}, $v_s = s+t_0$ and $d'_{s+t_0} \geq c_1c_2$, so
$$C_{s,\pi} \leq C \frac{\sup_{z \in E_s} \P_{s,z}(\tau_A > s+t_0)}{\P_{s,\pi}(\tau_A > s+t_0)},~~~~\forall s \geq 0, \forall \pi \in \cM_1(E_s),$$
where we set $C := \frac{1}{(c_1 c_2)^2}$. Thus, for any $x \in E_s$,
$$C_{t,\phi_{t,s}(\delta_x)} \leq C \frac{\sup_{z \in E_{t}} \P_{t,z}(\tau_A > t+t_0)}{\P_{t,\phi_{t,s}(\delta_x)}(\tau_A > t+t_0)}.$$
Now, the following lemma is needed : \begin{lemma} \label{petit} For any $s \leq t$ such that $t-s \geq t_0$, for any $x \in E_s$,
\begin{equation} \label{aim} \phi_{t,s}(\delta_x) \geq c_1 \nu_{t}.\end{equation}
In particular, the condition (A'1) holds replacing $t_0$ by any time $t_1$ greater than $t_0$. \end{lemma}
The proof of this lemma is postponed at the end of this proof. Then, by Lemma \ref{petit} and using (A'2),
\begin{align*}\P_{t,\phi_{t,s}(\delta_x)}(\tau_A > t+t_0) &\geq c_1 \P_{t,\nu_{t}}(\tau_A > t+t_0)\\  &\geq c_1 c_2 \sup_{z \in E_{t}} \P_{t,z}(\tau_A > t+t_0).\end{align*}
As a result $C_{t,\phi_{t,s}(\delta_x)} \leq 1/(c_1c_2)^3$, and 
$$\left| \P_{s,x}(X_{[s,t]} \in B | \tau_{A} > T)- \Q_{s,x}(X_{[s,t]} \in B) \right| \leq  \frac{1}{(c_1c_2)^3} (1-c_1c_2)^{ \left\lfloor \frac{T-t}{t_{0}} \right\rfloor }.$$
This concludes the first step.
\item : Convergence towards the quasi-ergodic distribution

We just proved that for any $0 \leq s \leq t$ and $x \in E_0$,
\begin{equation*}
||\P_x(X_{s} \in \cdot | \tau_{A} > t) - \Q_{0,x}(X_{s} \in \cdot) ||_{TV} \leq \frac{1}{(c_1c_2)^3 } (1-c_1c_2)^{ \left\lfloor \frac{t-s}{t_{0} }\right\rfloor  }.
\end{equation*}
Note that, in the same way, it was possible to consider a general initial law $\mu$ instead of a Dirac measure $\delta_x$, so that the inequality  
\begin{equation}
||\P_\mu(X_{s} \in \cdot | \tau_{A} > t) - \Q_{0,\eta_0 * \mu}(X_{s} \in \cdot) ||_{TV} \leq \frac{1}{(c_1c_2)^3 } (1-c_1c_2)^{ \left\lfloor \frac{t-s}{t_{0} }\right\rfloor  }
\label{expo-inequality}
\end{equation}
holds for any probability measure $\mu$ on $E_0$ (the notation $\eta_0 * \mu$ is defined in \eqref{notation}). 
As a result for any $0 \leq s \leq t$, for any $\mu \in \cM_1(E_0)$,
\begin{align*}
&\left|\left|\frac{1}{t} \int_0^t \P_\mu(X_{s} \in \cdot | \tau_{A} > t)ds - \frac{1}{t} \int_0^t \Q_{0,\eta_0*\mu}(X_{s} \in \cdot)ds \right|\right|_{TV}\\ 
&~~~~~~~~~~~~~~\leq  \frac{1}{(c_1c_2)^3  t} \int_0^t (1-c_1c_2)^{ \left\lfloor \frac{t-s}{t_{0} }\right\rfloor }ds \\
&~~~~~~~~~~~~~~\leq  \frac{1}{(c_1c_2)^3  t} \int_0^t (1-c_1c_2)^{  \frac{t-s}{t_{0} } -1}ds \\
&~~~~~~~~~~~~~~= \left(- \frac{t_0}{(c_1c_2)^3(1-c_1c_2) \log(1-c_1c_2) }\right) \times \frac{1-(1-c_1c_2)^{\frac{t}{t_{0}}}}{t}.
\end{align*}
Let $\beta$ as defined in \eqref{beta}. Then for any $\mu \in \cM_1(E_0)$ and $f \in \cB(E)$, 
\begin{align*}
  &\left|\frac{1}{t} \int_0^t \E_\mu(f(X_{s}) | \tau_{A} > t)ds - \beta(f) \right| \\
&\leq \left|\left|\frac{1}{t} \int_0^t \P_\mu(X_{s} \in \cdot | \tau_{A} > t)ds - \frac{1}{t} \int_0^t \Q_{0,\eta_0*\mu}(X_{s} \in \cdot)ds \right|\right|_{TV} + \left|\frac{1}{t} \int_0^t \E^\Q_{0,\eta_0*\mu}(f(X_{s}))ds -  \beta(f) \right| \\
&\leq \left(-\frac{t_0}{(c_1c_2)^3(1-c_1c_2) \log(1-c_1c_2)}\right) \times \frac{1-(1-c_1c_2)^{\frac{t}{t_{0}}}}{t}\\ &~~~~~~~~+ \left|\frac{1}{t} \int_0^t \E^\Q_{0,\eta_0*\mu}(f(X_{s}))ds -  \beta(f) \right|,
\end{align*}
where $\E^{\Q}_{0,\eta_0*\mu}$ is the expectation with respect to $\Q_{0,\eta_0*\mu}$.
Then, using the ergodic theorem for the $Q$-process,
$$\left|\frac{1}{t} \int_0^t \E_\mu(f(X_{s}) | \tau_{A} > t)ds - \beta(f) \right|  \underset{t \to \infty}{\longrightarrow} 0.$$
\end{enumerate}
\end{proof}

\begin{proof}[Proof of Lemma \ref{petit}]
Applying the condition (A'1) to the starting time $t-t_0$, 
$$\P_{t-t_0,y}(X_{t} \in \cdot, \tau_A > t) \geq c_1 \nu_{t}(\cdot) \P_{t-t_0,y}(\tau_A > t),~~~~\forall y \in E_{t-t_0}.$$
Then, for any probability measure $\mu$, integrating the last inequality over $\mu(dx)$ and dividing by $\P_{t-t_0,\mu}(\tau_A > t)$, one obtains
$$\P_{t-t_0,\mu}(X_{t} \in \cdot | \tau_A > t) \geq c_1 \nu_{t},~~~~\forall \mu \in \cM_1(E_{t-t_0}).$$
Hence, using the semi-flow property \eqref{semi-flow-property} of $(\phi_{t,s})_{s \leq t}$, for any $s \geq 0$ and $t \geq s+t_0$,
$$\phi_{t,s}(\delta_x) = \phi_{t,t-t_0} \circ \phi_{t-t_0,s}(\delta_x) = \P_{t-t_0,\phi_{t-t_0,s}(\delta_x)}(X_t \in \cdot | \tau_A > t) \geq c_1 \nu_{t},$$
which is \eqref{aim}.
\end{proof}

\begin{remark}
The time-homogeneity of the Markov process $(X_t)_{t \in I}$ does not play a particular role in the previous proof. In particular, Theorem \ref{expo-conv-thm} can be applied to time-inhomogeneous Markov process. However, in the next section, the time-homogeneity of $(X_t)_{t \in I}$ will be needed.
\end{remark}

\section{Some behaviors of moving boundaries and quasi-ergodicity}
In this section, we will focus on two types of behavior for the moving boundaries

\begin{enumerate}
\item when $A$ is $\gamma$-periodic with $\gamma > 0$;
\item when $A$ is non-increasing and converges at infinity towards $A_\infty \ne \emptyset$.
\end{enumerate}
Under Assumption (A'), the existence of the $Q$-process is provided by Theorem \ref{theorem-CV} (Theorem 3.3, \cite{CV2016}) and we get moreover an exponential convergence towards the $Q$-process provided by Theorem \ref{expo-conv-thm}. Now we want to investigate on the existence of a quasi-ergodic distribution in the two cases described above.   
\subsection{Quasi-ergodic distribution when $A$ is $\gamma$-periodic}
\label{subsection-per}

In this subsection, we will work on periodic moving boundaries and we will assume that the Markov process $(X_t)_{t \geq 0}$ satisfies the Assumption (A'). In particular, considering Assumption (A') for $s=0$, for any $x \in E_0$ and $t \in I$,
\begin{enumerate}
\item $\P_x(X_{t_0} \in \cdot | \tau_A > t_0) \geq c_{1} \nu_{t_0}$;
\item $\P_{\nu_0}(\tau_A > t) \geq c_{2} \P_x(\tau_A > t)$.
\end{enumerate}
As the Lemma \ref{petit} claims, any time $t_1$ greater than $t_0$ is suitable for the condition (A'1). Hence, without loss of generality, $t_0$ will be taken such that $t_0 = n_0 \gamma$ with $n_0 \in \N$.  
Moreover, by periodicity of $A$, it is easy to see that $(\nu_s)_{s \geq 0}$ can be chosen as a $\gamma$-periodic sequence. As a result, one has
$$\nu_{t_0} = \nu_{n_0 \gamma} = \nu_0.$$
In all what follows, we will consider such a choice of $(\nu_s)_{s \geq 0}$.
The aim is to obtain the convergence of $\frac{1}{t} \int_0^t \P_\mu(X_s \in \cdot | \tau_A > t)ds$ towards a quasi-ergodic distribution which will be unique. Let us state the following result, which is the more precise version of Theorem \ref{qed-presentation} introduced in Section \ref{general-results} :

\begin{theorem}
\label{qed}
Assume $A$ is $\gamma$-periodic with $\gamma > 0$, and assume that Assumption (A') is satisfied.
Then for any $\mu \in \cM_1(E_0)$,
$$\frac{1}{t} \int_0^t \P_\mu(X_s \in \cdot | \tau_{A} > t)ds  \underset{t \to \infty}{\overset{(d)}{\longrightarrow}}  \frac{1}{\gamma} \int_0^\gamma \Q_{0,\beta_\gamma}(X_s \in \cdot) ds, $$
where $\beta_\gamma$ is the invariant measure of $(X_{n \gamma})_{n \in \N}$ under $\Q_{0,\cdot}$, i.e.
$$\forall n \in \N,~~~~\beta_\gamma = \Q_{0,\beta_\gamma}(X_{n \gamma} \in \cdot) = \int_{E_0} \beta_\gamma(dx) \Q_{0,x}(X_{n \gamma} \in \cdot).$$
\end{theorem}
\begin{proof}[Proof of Theorem \ref{qed}]
We want to show an ergodic theorem for the time-inhomogeneous Markov process $(X_t)_{t \geq 0}$ under $(\Q_{s,x})_{s \geq 0, x \in E_s}$. 
Since $(A_t)_{t \geq 0}$ is $\gamma$-periodic, for any $0 \leq s \leq t$, for any $x \in E_s$,
\begin{equation}
\label{periodicity2}
\Q_{s+k\gamma,x}(X_{t + k \gamma} \in \cdot) = \Q_{s,x}(X_{t} \in \cdot),~~\forall k \in \Z_+.
\end{equation}
Moreover, for any $n \in \Z_+$,
\begin{align*}
\Q_{0,x}(X_{n \gamma} \in \cdot) &= \lim_{t \to \infty} \P_x(X_{n \gamma} \in \cdot | \tau_A > t) \\
&=  \lim_{ m \in \Z_+, m \to \infty} \P_x(X_{n \gamma} \in \cdot | \tau_A > m \gamma) \\
&= \lim_{ m \in \Z_+, m \to \infty} \P_{x}(Y_{n} \in \cdot | \tau_\d > m),
\end{align*}
where $\tau_\d$ is defined by
$$
\tau_\d= \left\{
    \begin{array}{ll}
       \inf\{n \geq 1 : \exists t \in ((n-1)\gamma, n\gamma], X_t \in A_t\} &\text{ if  } Y_0 \in E_0 \\
        0  & \text{ if   } Y_0 \in A_0
    \end{array}
\right.
$$
and $(Y_n)_{n \in \Z_+}$ is the time-homogeneous Markov chain defined by
$$Y_n =\left\{
\begin{array}{ll} X_{n\gamma} \text{ for } n < \tau_\d \\
\d ~~~~\text{      otherwise}
\end{array}
\right.$$
where $\d$ plays the role of an absorbing state for $(Y_n)_{n \in \Z_+}$.
In other words, $\tau_\d$ is an absorbing time for $(Y_n)_{n \in \Z_+}$ and, under $(\Q_{0,x})_{x \in E_0}$, the chain $(X_{n \gamma})_{n \in \Z_+}$ is the $Q$-process of  $(Y_{n})_{n \in \Z_+}$. \\ By Assumption (A') and recalling that we chose $(\nu_s)_{s \geq 0}$ as $\gamma$-periodic, $(Y_n)_{n \in \Z_+}$ satisfies the following Champagnat-Villemonais type condition : 
\begin{enumerate}
\item
$$\forall x \in E_0,~~~~\P_x(Y_{n_0} \in \cdot | \tau_\d > n_0) \geq c_{1}\nu_0;$$
\item
$$\forall x \in E_0, \forall n \in \Z_+, ~~~~\P_{\nu_0}(\tau_\d > n) \geq c_{2} \P_x(\tau_\d > n).$$
\end{enumerate}
where we recall that $n_0 = \frac{t_0}{\gamma}$.
Hence, by Theorem 3.1 in \cite{CV2014}, there exists $\beta_\gamma \in {\cal M}_1(E_0)$, $C > 0$ and $\rho \in (0,1)$ such that for any $n \in \Z_+$,
$$||\Q_{0,x}(X_{n \gamma} \in \cdot) - \beta_\gamma ||_{TV} \leq C\rho^n,~~\forall x \in E_0.$$
This implies that, under $\Q_{0,\cdot}$, $(X_{n \gamma})_{n \in \N}$ is Harris recurrent. We can therefore apply Theorem 2.1 in \cite{HK2010} and deduce that, for any nonnegative function $f$,
$$\frac{1}{t} \int_0^t f(X_s)ds \underset{t \to \infty}{\longrightarrow} \E^\Q_{0,\beta_\gamma}\left(\frac{1}{\gamma} \int_0^\gamma f(X_s)ds\right),~~\Q_{0,x}\text{-almost surely,} ~~\forall x \in E_0,$$
where $\E_{0,\mu}^\Q(G) = \int G d\Q_{0,\mu}$ for any measurable nonnegative function $G$ and $\mu \in {\cal M}_1(E_0)$.
It extends to $f \in {\cal B}(E)$ using $f = f_+ - f_-$ with $f_+, f_-$ non negative functions. 
Thus, by bounded Lebesgue's convergence theorem, for any $x \in E_0$ and for any $f \in {\cal B}(E)$,
$$\frac{1}{t} \int_0^t \E^{\Q}_{0,x}(f(X_s))ds \underset{t \to \infty}{\longrightarrow}  \E^\Q_{0,\beta_\gamma}\left(\frac{1}{\gamma} \int_0^\gamma f(X_s)ds\right).$$  
Hence the condition \eqref{beta} is satisfied. We conclude the proof using the second part of Theorem \ref{expo-conv-thm}.

\end{proof}
\begin{remark}
In \cite{HK2010}, Höpfner and Kutoyants claimed their results for Markov processes with continuous paths. It is easy to see using their arguments that the statement in Theorem 2.1. can be generalized to any time-inhomogeneous Markov processes $(X_t)_{t \in I}$ such that the condition of periodicity \eqref{periodicity2} is satisfied and the chain $(X_{n \gamma})_{n \in \Z_+}$ is Harris recurrent. See also Proposition 5 of \cite{hopfner2016ergodicity1}.
\end{remark}

\subsection{Quasi-ergodic distribution when $A$ converges at infinity}
\label{subsection-conv}
In this subsection, we assume that  $A$ is non-increasing and let $A_\infty$ as defined in \eqref{A_inf}. 
In what follows, we will first state the existence and uniqueness of a quasi-limiting distribution under these assumptions. Then we will deal with quasi-ergodic distribution.
\subsubsection{Quasi-limiting distribution}
First we state the following proposition which will be useful to prove the theorem on the existence and the uniqueness of the quasi-limiting distribution.
\begin{proposition}
\label{prop-qld}
Under Assumptions (A'),
for any $B \in {\cal E}$, the quantities 
$$\limsup_{t \to \infty} \P_{s,\mu}(X_t \in B | \tau_{A } > t) \text{ and }  \liminf_{t \to \infty} \P_{s,\mu}(X_t \in B | \tau_{A } > t)$$
do not depend on any couple $(s,\mu)$ such that $\mu \in {\cal M}_1(E_s)$.
\end{proposition}
\begin{proof}[Proof of Proposition \ref{prop-qld}]
 We recall the statement of \cite[Theorem 2.1]{CV2016}, which is adapted to our case :
\begin{theorem}[Theorem 2.1., \cite{CV2016}]
For any $s \in I$, for any $\mu_1, \mu_2 \in {\cal M}_1(E_s)$, for any $t \geq s + t_0$,
\begin{equation}
\label{mixing-property}|| \P_{s,\mu_1}(X_{t} \in \cdot | \tau_{A } > t) - \P_{s,\mu_2}(X_{t} \in \cdot | \tau_{A } > t)||_{TV} \leq 2  (1 - c_1c_2)^{\left\lfloor \frac{t-s}{t_0} \right\rfloor}.
\end{equation}
\end{theorem}

Let $B \in {\cal E}$. First we remark that, for $s$ fixed, $\limsup_{t \to \infty} \P_{\mu}(X_t \in B | \tau_{A \circ \theta_s} > t)$ does not depend on $\mu \in {\cal M}_1(E_s)$. This is straightforward since, thanks to \eqref{mixing-property}, for any $s \geq 0$ and any $\mu_1,\mu_2 \in {\cal M}_1(E_s)$,
$$||\P_{s,\mu_1}(X_t \in \cdot | \tau_{A } > t) - \P_{s,\mu_2}(X_t \in \cdot | \tau_{A } > t)||_{TV} \underset{ t \to 0}{\longrightarrow} 0,$$
which implies that, for any $s \geq 0$ and $\mu_1,\mu_2 \in {\cal M}_1(E_s)$,
\begin{equation}
\label{limsup}
\limsup_{t \to \infty}\P_{s,\mu_1}(X_t \in B | \tau_{A } > t) = \limsup_{t \to \infty} \P_{s,\mu_2}(X_t \in B | \tau_{A } > t).
\end{equation}
Now for any $u \geq 0$, recalling the notation $\phi_{t,s}(\mu) = \P_{s,\mu}(X_{s+u} \in \cdot | \tau_{A} >s+ u)$ for any $s \leq t$ and $\mu \in {\cal M}_1(E_s)$, 
\begin{align}
\limsup_{t \to \infty} \P_{s,\mu}(X_{t} \in B | \tau_{A } > t) &= \limsup_{t \to \infty} \P_{s,\mu}(X_{t + u} \in B | \tau_{A} > t + u)  \notag \\
 &= \limsup_{t \to \infty} \P_{s+u,\phi_{u+s,s}(\mu)}(X_{t} \in B | \tau_{A } > t) \notag \\
\label{trez}
&= \limsup_{t \to \infty} \P_{s+u,\nu}(X_{t} \in B | \tau_{A } > t),\end{align}
where we used first the semi-flow property of $(\phi_{t,s})_{s \leq t}$, and then \eqref{limsup} with a given probability measure $\nu \in {\cal M}_1(E_{s+u})$.\\
Hence \eqref{trez} show that $\limsup_{t \to \infty} \P_{s,\mu}(X_t \in B | \tau_{A} > t)$ does not depend on any couple $(s, \mu)$ satisfying $s \in I$ and $\mu \in {\cal M}_1(E_s)$. 
A similar reasoning shows that $\liminf_{t \to \infty} \P_{s,\mu}(X_t \in B | \tau_{A} > t)$ does not depend on $s$ and $\mu$ either.

\end{proof} 
Before showing the existence of a quasi-limiting and a quasi-ergodic distribution, let us state the following proposition providing a uniform-in-time convergence of the time-inhomogeneous conditioned semi-group towards the time-homogeneous limit semi-group.
\begin{proposition}
\label{uc}
Assume $(H_{hom})$, $(H_\infty)$ and $(H'_\infty)$, and let $s_0$ and $x_0$ as defined in $(H'_\infty)$. Then,
\begin{equation}\label{uniform-convergence}\lim_{s \to \infty} \sup_{0 \leq t \leq T} ||\P_{s,x_0}(X_{t+s} \in \cdot | \tau_{A} >s + T) - \P_{x_0}(X_t \in \cdot | \tau_{A_\infty} > T)||_{TV} =0.\end{equation}
\end{proposition}
\begin{remark}   
Taking $T = t$, \eqref{uniform-convergence} implies that
\begin{equation*}\lim_{s \to \infty} \sup_{t \geq 0} ||\P_{s,x_0}(X_{t+s} \in \cdot | \tau_{A} >s+ t) - \P_{x_0}(X_t \in \cdot | \tau_{A_\infty} > t)||_{TV} =0.\end{equation*} 
This is actually a stronger version than the definition of asymptotic pseudotrajectories as introduced by Benaïm and Hirsch in \cite{BH1996}, for which the supremum is usually only taken on a compact set of time. In a practical way, it is difficult to use the weak version to show the convergence of the time-inhomogeneous semi-flow; considering instead a uniform convergence on $\R_+$ will be useful for our purpose. The interested reader can see \cite{benaim1999} for more details about asymptotic pseudotrajectories.  
\end{remark}
\begin{proof}[Proof of Proposition \ref{uc}]
For any $0 \leq t \leq T$, $s \geq s_0$ and $B \in {\cal E}$,
\begin{align*}
&|\P_{s,x_0}(X_{t+s} \in B | \tau_{A} > s+T) - \P_{x_0}(X_t \in B | \tau_{A_\infty} > T)| \\ &=|\P_{s,x_0}(X_{t+s} \in B | \tau_{A } >s+ T) - \P_{s,x_0}(X_{s+t} \in B | \tau_{A_\infty} > s+T)| \\ &= \left| \frac{\P_{s,x_0}(\tau_{A_\infty} >s+ T)}{\P_{s,x_0}(\tau_{A } >s+ T)} \frac{\P_{s,x_0}(X_{t+s} \in B, \tau_{A } >s+  T)}{\P_{s,x_0}(\tau_{A_\infty} >s+T)} - \frac{\P_{s,x_0}(X_{s+t} \in B, \tau_{A_\infty} >s+T)}{\P_{s,x_0}(\tau_{A_\infty} >s+T)} \right| \\
&\leq \left| \frac{\P_{s,x_0}(\tau_{A_\infty} >s+T)}{\P_{s,x_0}(\tau_{A } >s+T)} \frac{\P_{s,x_0}(X_{s+t} \in B, \tau_{A} >s+ T)}{\P_{s,x_0}(\tau_{A_\infty} >s+T)} - \frac{\P_{s,x_0}(X_{t+s} \in B, \tau_{A } >s+T)}{\P_{s,x_0}(\tau_{A_\infty} >s+T)} \right| \\
&~~~~~~+ \left| \frac{\P_{s,x_0}(X_{s+t} \in B, \tau_{A} >s+T)}{\P_{s,x_0}(\tau_{A_\infty} > s+T)} - \frac{\P_{s,x_0}(X_{s+t} \in B, \tau_{A_\infty} >s+T)}{\P_{s,x_0}(\tau_{A_\infty} >s+T)} \right| \\
&\leq \left| \frac{\P_{s,x_0}(\tau_{A_\infty} >s+ T)}{\P_{s,x_0}(\tau_{A} >s+ T)}  - 1 \right| \times \frac{\P_{s,x_0}(X_{s+t} \in B, \tau_{A} >s + T)}{\P_{s,x_0}(\tau_{A_\infty} >s+T)} \\
&~~~~~~+ \left| \frac{\P_{s,x_0}(X_{s+t} \in B, \tau_{A} >s+T)-\P_{s,x_0}(X_{s+t} \in B, \tau_{A_\infty} > s+T)}{\P_{s,x_0}(\tau_{A_\infty} > s+T)}\right| \\
&\leq \left| \frac{\P_{s,x_0}(\tau_{A_\infty} >s+T)}{\P_{s,x_0}(\tau_{A } >s+T)}  - 1 \right|+  \frac{\P_{s,x_0}(\tau_{A } \leq s+T < \tau_{A_\infty})}{\P_{s,x_0}(\tau_{A_\infty} > s+T)},
\end{align*}
where we used several times the fact that $A_\infty \subset A_t$ for any $t$ (in particular to say that $\P_{s,x_0}(\tau_{A }>s+u) \leq \P_{s,x_0}(\tau_{A_\infty} >s+ u)$ for any $u \geq 0$). Hence it is enough to prove that 
\begin{equation}
\label{dini}
\sup_{t \geq 0} \frac{\P_{s,x_0}(\tau_{A } \leq s+t < \tau_{A_\infty})}{\P_{x_0}(\tau_{A_\infty} > t)} \underset{s \to \infty}{\longrightarrow} 0.
\end{equation}
As a matter of fact, \eqref{dini} is equivalent to
$$
\sup_{t \geq 0} \left| \frac{\P_{s,x_0}( \tau_{A} >s+ t)}{\P_{x_0}(\tau_{A_\infty} > t)} - 1 \right|  \underset{s \to \infty}{\longrightarrow} 0.$$
and it is easy to check that, for general functions $(s,t) \to f(s,t)$, $(f(s,\cdot))_{s \geq 0}$ converges uniformly towards the constant function equal to $1$ if and only if $\left(\frac{1}{f(s,\cdot)}\right)_{s \geq 0}$ also converges uniformly towards $1$. \\ 
Fix $t \geq 0$. Since $A$ is non-increasing, for any $s < s'$,
$$\frac{\P_{s,x_0}(\tau_{A} \leq s+t < \tau_{A_\infty})}{\P_{x_0}(\tau_{A_\infty} > t)} \geq \frac{\P_{s',x_0}(\tau_{A} \leq s'+t < \tau_{A_\infty})}{\P_{x_0}(\tau_{A_\infty} > t)}.$$ 
Moreover, using \textit{the convergence in law for the hitting times} of Assumption $(H_{hom})$,  one has, for any $t \geq 0$,
$$\frac{\P_{s,x_0}(\tau_{A } \leq s+t < \tau_{A_\infty})}{\P_{x_0}(\tau_{A_\infty} > t)} \underset{s \to \infty}{\longrightarrow} 0.$$
Finally, by \textit{the strong Markov property} of Assumption $(H_{hom})$, for any $t \geq 0$,
\begin{align*}
\label{markov-property}
\P_{s,x_0}(\tau_{A } \leq s+t < \tau_{A_\infty}) &= \E_{s,x_0}(\1_{\tau_{A } \leq s+t} \phi(X_{\tau_{A }}, \tau_{A }-s,t)),
\end{align*}
where $\phi(\cdot,\cdot, \cdot)$ is defined as follows
$$\forall z \in E_\infty, \forall~0 \leq u \leq t,~~~~\phi(z,u,t) = \P_z(\tau_{A_\infty} > t-u).$$
In \cite{CV2017} it is shown (Theorem 2.1.) that, under $(H_\infty)$,  there exists a constant $a_1 > 0$ such that, for any $x \in E_\infty$ and $t \in I$,
\begin{equation}
\label{eta}
\left| \frac{\eta_\infty(x)}{e^{\lambda_\infty t}\P_x(\tau_{A_\infty} > t)} - 1 \right| \leq a_1 e^{-\gamma_\infty t},
\end{equation}
where $\lambda_\infty$, $\gamma_\infty$ and the function $\eta_\infty$ were defined respectively in \eqref{lambda}, \eqref{gamma} and \eqref{def-eta} in the section \ref{general-results}.
This implies therefore that there exists $C > 0$ such that, for any $x \in E_\infty$ and $t \in I$,
\begin{equation*}
\left| \eta_\infty(x) - e^{\lambda_\infty t} \P_x(\tau_{A_\infty} > t) \right| \leq a_1 e^{-\gamma_\infty t} e^{\lambda_\infty t} \P_x(\tau_{A_\infty} > t) \leq C e^{- \gamma_\infty t},
\end{equation*}
where we used that the function $t \mapsto e^{\lambda_\infty t} \P_x(\tau_{A_\infty} > t)$ is upper bounded uniformly in $x$.
Hence, for any $x \in E_\infty$ and $t \geq 0$,
\begin{equation}
\label{eta2}
\frac{ e^{\lambda_\infty t} \P_x(\tau_{A_\infty} > t)}{\eta_\infty(x)} \leq  1 + \frac{C e^{- \gamma_\infty t}}{\eta_\infty(x)}.
\end{equation} 
By \cite[Proposition 2.3]{CV2014}, $\1_{\tau_{A } \leq t+s} \frac{\phi(X_{\tau_{A }}, \tau_{A }-s,t)}{\P_{x_0}(\tau_{A_\infty} > t)}$ converges almost surely towards $ e^{\lambda_\infty (\tau_{A }-s)} \frac{\eta_\infty(X_{\tau_{A }})}{\eta_\infty(x_0)}$, and using \eqref{eta} and \eqref{eta2}, 
\begin{multline*}
\frac{\phi(X_{\tau_{A }}, \tau_{A }-s,t)}{\P_{x_0}(\tau_{A_\infty} > t)} \\ = e^{\lambda_\infty (\tau_{A }-s)} \frac{\eta_\infty(X_{\tau_{A }})}{\eta_\infty(x_0)} \left[ \frac{\eta_\infty(x_0)}{e^{\lambda_\infty t}\P_{x_0}(\tau_{A_\infty} > t)} \times  \1_{\tau_{A } \leq t+s} \frac{e^{\lambda_\infty (t - (\tau_A-s))}\phi(X_{\tau_{A }}, \tau_{A }-s,t)}{\eta_\infty(X_{\tau_A})} \right]. \\
\leq e^{\lambda_\infty (\tau_{A }-s)} \frac{\eta_\infty(X_{\tau_{A }})}{\eta_\infty(x_0)} (1 + a_1 e^{- \gamma_\infty t}) \left(1 + \frac{C}{\eta_\infty(x_0)} e^{-\gamma_\infty (t- (\tau_A-s))}\1_{\tau_A \leq t+s}\right) \\
\leq (1+a_1)\left(1 + \frac{C}{\eta_\infty(x_0)}\right)\frac{\eta_\infty(X_{\tau_A})}{\eta_\infty(x_0)} e^{\lambda_\infty (\tau_A-s)}.
\end{multline*}
 Then, under $(H_\infty')$, by the bounded Lebesgue's convergence theorem, for any $s \geq s_0$,
\begin{align*}
\lim_{t \to \infty} \frac{\P_{s,x_0}(\tau_{A } \leq s+t < \tau_{A_\infty})}{\P_{x_0}(\tau_{A_\infty} > t)} &= \lim_{t \to \infty} \E_{s,x_0}\left(\1_{\tau_{A } \leq s+t} \frac{\phi(X_{\tau_{A }}, \tau_{A }-s,t)}{\P_{x_0}(\tau_{A_\infty} > t)}\right) \\
&=\E_{s,x_0} \left(e^{\lambda_\infty (\tau_{A }-s)} \frac{\eta_\infty(X_{\tau_{A }})}{\eta_\infty(x_0)}\right).
\end{align*}
For any $s \geq 0$, we can therefore define $f_s : t \to \frac{\P_{s,x_0}(\tau_{A } \leq s+t < \tau_{A_\infty})}{\P_{x_0}(\tau_{A_\infty} > t)}$ on the Alexandroff extension $\R_+ \cup \{\infty\}$ setting
$$f_s(\infty) :=  \lim_{t \to \infty} \frac{\P_{s,x_0}(\tau_{A} \leq s+t < \tau_{A_\infty})}{\P_{x_0}(\tau_{A_\infty} > t)} = \E_{s,x_0} \left(e^{\lambda_\infty (\tau_{A }-s)} \frac{\eta_\infty(X_{\tau_{A }})}{\eta_\infty(x_0)}\right).$$  
Then, like any $t \in \R$, $\left(f_s(\infty)\right)_{s \geq 0}$ is non-increasing and, by the assumption $(H_\infty')$,
$$\lim_{s \to \infty} f_s(\infty) = \lim_{s \to \infty} \E_{s,x_0} \left(e^{\lambda_\infty (\tau_{A }-s)} \frac{\eta_\infty(X_{\tau_{A }})}{\eta_\infty(x_0)}\right) = 0.$$
We conclude to the uniform convergence \eqref{dini} using Dini's theorem for a non-increasing sequence of functions.
\end{proof}
Now one will prove Theorem \ref{first-thm} stated in Section \ref{general-results}, which is recalled below : 
\begin{theorem}
\label{qld}
Under Assumptions (A'), $(H_{hom})$, $(H_\infty)$ and $(H'_\infty)$,
for any $\mu \in {\cal M}_1(E_0)$, 
$$\P_\mu(X_t \in \cdot | \tau_{A} > t) \underset{t \to \infty}{\overset{(d)}{\longrightarrow}} \alpha_\infty,$$
where $\alpha_\infty$ is the quasi-stationary distribution defined in Assumption $(H_\infty)$. 
\end{theorem}
\begin{proof}[Proof of Theorem \ref{qld} (Theorem \ref{first-thm})]
Fix $B \in {\cal E}$ and note that, by Assumption $(H_\infty)$,  for any $\mu \in {\cal M}_1(E_\infty)$,
$$\limsup_{t \to \infty} \P_\mu(X_{t} \in B | \tau_{A_\infty} > t) = \liminf_{t \to \infty} \P_\mu(X_{t} \in B | \tau_{A_\infty} > t) = \alpha_\infty(B),$$
where we recall that $\alpha_\infty$ is the quasi-stationary distribution of $(X_t)_{t \in I}$ absorbed at $A_\infty$.
By Proposition \ref{prop-qld}, for a given $s \in I$, $\limsup_{t \to \infty} \P_{s,\mu}(X_{t} \in B | \tau_{A} > t)$ and $\liminf_{t \to \infty} \P_{s,\mu}(X_{t} \in B | \tau_{A } > t)$ do not depend on $\mu \in {\cal M}_1(E_s)$. Denote therefore by $F_{sup}$ and $F_{inf}$ the functions defined by, for any $s \geq s_0$ and any $\mu \in {\cal M}_1(E_s)$ ,
$$F_{sup}(s) := \limsup_{t \to \infty} \P_{{s,\mu}}(X_{s+t} \in B | \tau_{A } > s+t) = \limsup_{t \to \infty} \P_{s,x_0}(X_{s+t} \in B | \tau_{A } > s+t)$$
and
$$F_{inf}(s) :=  \liminf_{t \to \infty} \P_{{s,\mu}}(X_{s+t} \in B | \tau_{A } > s+t) =  \liminf_{t \to \infty} \P_{s,x_0}(X_{s+t} \in B | \tau_{A} > s+t)$$
where $x_0$ is defined as in $(H'_\infty)$. 
Then $F_{sup}$ and $F_{inf}$ do not depend on $s$ either (by Proposition \ref{prop-qld}), hence for any $s \geq 0$,
$$F_{sup}(s) = \lim_{u \to \infty} F_{sup}(u),$$
and
$$F_{inf}(s) = \lim_{u \to \infty} F_{inf}(u).$$
Moreover, by the uniform convergence \eqref{uniform-convergence} of Proposition \ref{uc},
\begin{align*}
 \lim_{u \to \infty} F_{sup}(u) &=  \lim_{u \to \infty} \limsup_{t \to \infty} \P_{u,x_0}(X_{u+t} \in B | \tau_{A } > u+t) \\
&= \limsup_{t \to \infty} \P_{x_0}(X_{t} \in B | \tau_{A_\infty} > t) \\
&= \alpha_\infty(B).
\end{align*}
Similarly,
$$ \lim_{u \to \infty} F_{inf}(u) = \alpha_\infty(B).$$ 
Hence, for any $s \geq 0$ and $\mu \in {\cal M}_1(E_s)$,
$$\limsup_{t \to \infty} \P_{s,\mu}(X_{t} \in B | \tau_{A } > t) = \liminf_{t \to \infty} \P_{s,\mu}(X_{t} \in B | \tau_{A } > t) = \alpha_\infty(B).$$ 
\end{proof}
\begin{remark}
It can be interesting to compare this result and this proof with the one of \cite[Theorem 3.11]{BCG2017} obtained by Bansaye and al. In particular, they used a different property of asymptotic homogeneity, which is uniform-in-state in their case.
\end{remark}  
\subsubsection{Quasi-ergodic distribution}
Now we can state the existence and uniqueness of the quasi-ergodic distribution :
\begin{theorem}
\label{qed-converge}
Under the assumptions of Theorem \ref{qld}, 
for any $\mu \in \cM_1(E_0)$,
$$\frac{1}{t} \int_0^t \P_x(X_s \in \cdot | \tau_{A} > t)ds  \underset{t \to \infty}{\overset{(d)}{\longrightarrow}}  \beta_\infty,  $$
where $\beta_\infty$ is the unique invariant measure of the $Q$-process of $(X_t)_{t \geq 0}$ absorbed by $A_\infty$.

\end{theorem}
\begin{proof}[Proof of Theorem \ref{qed-converge}]

We will show that the $Q$-process converges weakly towards a probability measure. Fix $B \in {\cal E}$. Since we have the following inequality shown in Theorem 3.3 of \cite{CV2016}
$$||\Q_{s,\mu_1}(X_t \in \cdot) - \Q_{s,\mu_2}(X_t \in \cdot)||_{TV} \leq 2  (1-c_1c_2)^{\left\lfloor \frac{t-s}{t_0} \right\rfloor},$$ for any $\mu_1, \mu_2 \in \cM_1(E_s)$.   
We get therefore that
$$\limsup_{t \to \infty} \Q_{s,\mu_1}(X_t \in B)  = \limsup_{t \to \infty} \Q_{s,\mu_2}(X_t \in B),$$
and we can therefore use the reasoning of the proof of Proposition \ref{prop-qld} to show that, for any $s,u \in I$, for any $\mu, \nu \in {\cal M}_1(E_s) \times {\cal M}_1(E_{s+u})$,
$$\limsup_{t \to \infty} \Q_{s,\mu}(X_t \in B) = \limsup_{t \to \infty} \Q_{s+u,\nu}(X_t \in B).$$
In particular, for any $s \geq 0$, $\mu \in {\cal M}_1(E_{s})$,
$$\limsup_{t \to \infty} \Q_{s,\mu}(X_t \in B) = \lim_{u \to \infty} \limsup_{t \to \infty} \Q_{u,x_0}(X_t \in B).$$
By the uniform convergence \eqref{uniform-convergence} of Proposition \ref{uc}, for any $s \geq 0$, $\mu \in {\cal M}_1(E_{s})$,
\begin{align*}
\limsup_{t \to \infty} \Q_{s,\mu}(X_t \in B)  &= \lim_{u \to \infty} \limsup_{t \to \infty} \Q_{u,x_0}(X_{u+t} \in B) \\
&=  \lim_{u \to \infty}  \limsup_{t \to \infty} \lim_{T \to \infty} \P_{u,x_0}(X_{u+t} \in B | \tau_{A} > u+T) \\
&= \limsup_{t \to \infty} \lim_{T \to \infty} \P_{x_0}(X_t \in B | \tau_{A_\infty} > T) \\
&= \limsup_{t \to \infty} \Q^\infty_{x_0}(X_t \in B), \end{align*}
where, for any $x \in E_\infty$,
$$\Q^\infty_x(X_t \in B) = \lim_{T \to \infty} \P_x(X_t \in B | \tau_{A_\infty} > T)$$
is well-defined by \cite[Theorem 3.1]{CV2014} under Assumption $(H_\infty)$. This theorem states moreover that $(X_t)_{t \in I}$ admits under $(\Q^\infty_x)_{x \in E_\infty}$ a unique invariant measure $\beta_\infty$ and for any $x \in E_\infty,$ 
$$\lim_{t \to \infty} \Q^\infty_{x}(X_t \in \cdot) = \beta_\infty.$$
Thus, for any $B \in {\cal E}$, $s \geq 0$ and $x \in E_s$,
 \begin{align*}
\limsup_{t \to \infty} \Q_{s,x}(X_t \in B)  &= \beta_\infty(B) \\
&= \liminf_{t \to \infty} \Q_{s,x}(X_t \in B).
\end{align*}
Finally, thanks to the convergence in law of the $Q$-process we just prove, we can deduce the weak ergodic theorem using Cesaro's rule
$$\lim_{t \to \infty} \frac{1}{t} \int_0^t \Q_{0,x}(X_s \in \cdot)ds = \beta_\infty.$$
Hence the condition \eqref{beta} holds. As a result we can apply the second part of Theorem \ref{expo-conv-thm} and conclude the proof.
\end{proof}

\section{Example : Diffusion on $\R$}
Let $(X_t)_{t \geq 0}$ be a diffusion on $\R$ satisfying the following stochastic differential equation
\begin{equation}
\label{sde}
       dX_t = dW_t - V(X_t)dt,
\end{equation}
where $(W_t)_{t \in \R_+}$ is Brownian motion on $\R$ and $V \in {\cal C}^1(\R)$. We assume that, under $\P_x$, there exists a strongly unique non explosive solution of \eqref{sde} such that $X_0 = x$ almost surely.\\ 
Let $h$ be a positive bounded ${\cal C}^1$-function. We define $\tau_h$ the random time defined by 
$$\tau_h = \inf\{t \geq 0 : X_t \leq h(t)\}.$$
Let us also recall the definition of the semi-flow $(\phi_{t,s})_{s \leq t}$ when the absorbing boundary is $h$ :
$$\phi_{t,s} : \mu \mapsto \P_{s,\mu}(X_t \in \cdot | \tau_h > t),~~~~\forall s \leq t.$$

\subsection{Preliminaries on one-dimensional diffusion processes coming down from infinity}
\label{pre}
We assume that $(X_t)_{t \in \R_+}$ comes down from infinity (in the sense given in \cite{CCLMMSM2009}), that is, there exists $y > h_{max} := \sup_{s \geq 0} h(s)$ and $t > 0$ such that
\begin{equation}
\label{comes-down}
\lim_{x \to \infty} \P_x(\tau_y < t) > 0,
\end{equation}
where, for any $z \in \R$,
$$\tau_z := \inf\{t \geq 0 : X_t = z\}.$$
 In this case, as remarked in the subsection 4.5.2. of \cite{CV2017b}, $(X_t)_{t \geq 0}$ satisfies then 
$$\int_0^1 \frac{1}{x} \left(\sup_{y \in (z,\Lambda_z^{-1}(x)]} \frac{1}{\Lambda_z(y)} \int_z^y \Lambda_z(\xi)^2m(d\xi)\right)dx < \infty,$$
where, for any $z \geq 0$, $\Lambda_z$ is the scale function of $X$ satisfying $\Lambda_z(z) = 0$ and defined by
\begin{equation}
\label{scale-function}
\Lambda_z(x) = \int_z^x e^{2 \int_0^y V(\xi)d\xi} dy,~~~~\forall x \geq z
\end{equation}
 and $m$ is the speed measure of $(X_t)_{t \geq 0}$ defined by
$$m(d\xi) = 2 e^{-2\int_0^\xi V(\xi')d\xi'}d\xi.$$
In particular, for any $z \geq 0$, the process $Y^z := (\Lambda_z(X_t))_{t \geq 0}$ is a local martingale and, since $X$ is solution of \eqref{sde}, by Itô's formula, for any $t \geq 0$,
$$Y^z_t = Y^z_0 + \int_0^t \Lambda_z'(\Lambda^{-1}_z(Y^z_s))dW_s.$$
Note that $\Lambda_z' = \Lambda_0' =   e^{2 \int_0^\cdot V(\xi)d\xi}$ for any $z$. So denoting for any $x,z \geq 0$ $\sigma_z(x) := \Lambda_z'(\Lambda_z^{-1}) = \Lambda_0'(\Lambda_z^{-1})$, one has 
$$dY_t^z = \sigma_z(Y^z_t)dW_t.$$ 
 
Adapting \cite[Theorem 4.6]{CV2017b} for general diffusion processes, we deduce that for any $t > 0$, there exists $A^z_t < \infty$ such that

$$\P_x(t < \tau_z) \leq A^z_t \Lambda_z(x),~~~~\forall x \geq z.$$

So let $u_1 \geq 0$ arbitrarily chosen. One has for any $z \geq 0$,
\begin{equation}
\label{tempsdesortie}\P_x(u_1 < \tau_z) \leq A^z_{u_1} \Lambda_z(x),~~~~\forall x \geq z\end{equation}
or, equivalently,
$$\P_{\Lambda_z^{-1}(x)}(u_1 < \tau_z) \leq A^z_{u_1} x,~~~~\forall x \geq 0.$$
Denoting for any $r \geq 0$ and for any process $(R_t)_{t \geq 0}$ $\tau_r(R) := \inf\{t \geq 0 : R_t = r\}$, one has for any $z \geq 0$ and $x \geq r$,
$$\P_{\Lambda_z^{-1}(x)}(u_1 < \tau_{\Lambda_z^{-1}(r)}) = \P\left( \tau_r(Y^{z}) > u_1\middle| Y^z_0 = x\right).$$
Since $z \to \Lambda_z^{-1}(x)$ is increasing for any $x > 0$, then, for any $x > 0$ and for any $z \geq z'$,
\begin{equation}
\label{encadrement}
\sigma_z(x) \geq \sigma_{z'}(x).
\end{equation}
Thus, using the same reasoning as in the proof of Lemma 4.2. in the paper \cite{CV2016}, it is possible to show that \eqref{encadrement} implies that, for any $z \geq z'$ and $x \geq r$
$$ \P\left( \tau_r(Y^{z'}) > u_1\middle| Y^{z'}_0 = x\right) \geq  \P\left( \tau_r(Y^{z}) > u_1\middle| \tilde{Y}^z_0 = x\right)$$ or, equivalently,
\begin{equation}
\label{ineg-utile}
\P_{\Lambda_z^{-1}(x)}(u_1 < \tau_{\Lambda_z^{-1}(r)}) \leq \P_{\Lambda_{z'}^{-1}(x)}(u_1 < \tau_{\Lambda_{z'}^{-1}(r)}).\end{equation}
Taking $z'=r=0$, for any $x \geq 0$,
$$\P_{\Lambda_z^{-1}(x)}(u_1 < \tau_z) \leq \P_{\Lambda_{0}^{-1}(x)}(u_1 < \tau_{0}) \leq A_{u_1}^0x,~~~~\forall x \geq 0.$$
In conclusion, one has, for any $z \geq 0$, 
\begin{equation}
\label{tempsdesortie2}\P_x(u_1 < \tau_z) \leq A^0_{u_1} \Lambda_z(x),~~~~\forall x \geq z.\end{equation}
One set $A := A_{u_1}^0$. 
Let us now state and prove the following lemma.
\begin{lemma}
\label{essentiel}
There exists $u_0 \geq 0$, $\kappa > 0$ a family of probability measures $(\psi_z)_{z \in [0,h_{max}]}$ such that, for any $z \in [0,h_{max}]$,
\begin{equation}
\label{first-cond}
\P_x(X_{u} \in \cdot | \tau_{z} > u) \geq \kappa \psi_z,~~~~\forall x > z, \forall u > u_0.
\end{equation}
\end{lemma}
The difference between this lemma and  \cite[Theorem~4.1]{CV2017b} is that the time $u_0$ and the constant $\kappa$ do not depend on $z$. 
The sketch of the proof is inspired from the proof of the Theorem 4.1 presented in \cite[Subsection~5.1]{CV2017b}.  
\begin{proof}[Proof of Lemma \ref{essentiel}]
The following proof is divided into two steps.

\medskip\noindent\textit{Step 1. : Mimicking the Step 1 in the proof of \cite[Theorem~4.1]{CV2017b}} 

The aim of this first step is to prove that there exist $\epsilon, c > 0$ not depending on $z$ such that
\begin{equation}
\label{goal}
\P_x(\Lambda_z(X_{u_1}) \geq \epsilon | \tau_z > u_1) \geq c,~~~~\forall x > z.
\end{equation}
Since, for any $z \in [0,h_{max}]$, $\Lambda_z(X)$ is a local martingale, one has for any $x \in (z,\Lambda_z^{-1}(1))$,
\begin{align*}
\Lambda_z(x) &= \E_x(\Lambda_z(X_{u_1 \land \tau_z \land \tau_{\Lambda_z^{-1}(1)}}))\\ & = \P_x(\tau_z > u_1)\E_x(\Lambda_z(X_{u_1 \land\tau_{\Lambda_z^{-1}(1)}})|\tau_z > u_1) + \P_x(\tau_{\Lambda_z^{-1}(1)} < \tau_z \leq u_1).\end{align*}
By Markov property,
\begin{align*}
\P_x(\tau_{\Lambda_z^{-1}(1)} < \tau_z \leq u_1) &\leq \E_x(\1_{\tau_{\Lambda_z^{-1}(1)} < \tau_z \land u_1} \P_{\Lambda_z^{-1}(1)}(\tau_z \leq u_1)) \\
&\leq \P_x(\tau_{\Lambda_z^{-1}(1)} < \tau_z)  \P_{\Lambda_z^{-1}(1)}(\tau_z \leq u_1)\\
&= \Lambda_z(x)   \P_{\Lambda_z^{-1}(1)}(\tau_z \leq u_1),
\end{align*}
where the following identity is used 
$$\P_x(\tau_{a} < \tau_b) = \frac{\Lambda_z(x) - \Lambda_z(b)}{\Lambda_z(a) - \Lambda_z(b)},~~~~\forall x \in [a,b].$$ 
As a result, using \eqref{tempsdesortie2}, one has, for any $x \in (z,\Lambda_z^{-1}(1))$,
$$\E_x(1-\Lambda_z(X_{1 \land \tau_{\Lambda_z^{-1}(1)}}) | 1 < \tau_z) \leq 1 - \frac{1}{A_z'},$$
where $A'_z := A/\P_{\Lambda_z^{-1}(1)}(u_1 < \tau_z)$. But, since $z \in [0,h_{max}]$, the inequality \eqref{ineg-utile} applied to $r=0$ implies that 
$$\P_{\Lambda_z^{-1}(1)}(u_1 < \tau_z) \geq \P_{\Lambda_{h_{max}}^{-1}(1)}(u_1 < \tau_{h_{max}}).$$
So, defining $A' :=  A/ \P_{\Lambda_{h_{max}}^{-1}(1)}(u_1 < \tau_{h_{max}})$, one has 
  $$\E_x(1-\Lambda_z(X_{u_1 \land \tau_{\Lambda_z^{-1}(1)}}) | u_1 < \tau_z) \leq 1 - \frac{1}{A'},~~~~\forall x \in (z,\Lambda_1^{-1}(1)).$$
Thus, using Markov's inequality, 
$$\P_x\left(\Lambda_z(X_{u_1 \land \tau_{\Lambda_z^{-1}(1)}}) \leq \frac{1}{2A'-1} \middle| \tau_z > u_1\right) \leq 1 - \frac{1}{2A'}.$$
Then, since $A' > 1$ by \eqref{tempsdesortie2}, $1/(2A'-1) < 1$. Thus, for any $\epsilon \in (0,1/(2A'-1))$ and $ x \in (z,\Lambda_z^{-1}(1/(2A'-1)))$,
\begin{align*}
&\P_x(\Lambda_z(X_{u_1}) \geq \epsilon, \tau_z > u_1) \\&\geq  \P_x(\tau_{\Lambda_z^{-1}(1/(2A'-1))} < u_1 \land \tau_z, \tau_{\Lambda_z^{-1}(\epsilon)} \circ \theta_{\tau_{\Lambda_z^{-1}(1/(2A'-1))}} > u_1 + \tau_{\Lambda_z^{-1}(1/(2A'-1))}) \\ &= \P_x(\tau_{\Lambda_z^{-1}(1/(2A'-1))} < u_1 \land \tau_z) \P_{\Lambda_z^{-1}(1/(2A'-1))}(\tau_{\Lambda_z^{-1}(\epsilon)} > u_1) \\
&\geq \P_x\left(\Lambda_z(X_{u_1 \land \tau_{\Lambda_z^{-1}(1)} \land \tau_z}) \geq 1/(2A'-1)\right)  \P_{\Lambda_z^{-1}(1/(2A'-1))}(\tau_{\Lambda_z^{-1}(\epsilon)} > u_1) \\
&\geq \frac{\P_x(\tau_z > u_1)}{2A'}  \P_{\Lambda_z^{-1}(1/(2A'-1))}(\tau_{\Lambda_z^{-1}(\epsilon)} > u_1) \\
&\geq \frac{\P_x(\tau_z > u_1)}{2A'}  \P_{\Lambda_{h_{max}}^{-1}(1/(2A'-1))}(\tau_{\Lambda_{h_{max}}^{-1}(\epsilon)} > u_1), 
\end{align*}
where \eqref{ineg-utile} is used again.
So, if $\epsilon$ is chosen such that $\P_{\Lambda_{h_{max}}^{-1}(1/(2A'-1))}(\tau_{\Lambda_{h_{max}}^{-1}(\epsilon)} > u_1) > 0$ (it is possible since  $\P_{\Lambda_{h_{max}}^{-1}(1/(2A'-1))}(\tau_z > u_1) > 0$), then there exist $\epsilon \in (0,1/(2A'-1))$ and $c > 0$ (not depending on $z$) such that, for any $x \in (z,\Lambda_z^{-1}(1/(2A'-1)))$, 
    $$\P_x(\Lambda_z(X_{u_1}) \geq \epsilon | \tau_z > u_1) \geq c.$$
 For $x \geq \Lambda_{z}^{-1}(1/(2A'-1))$, 
\begin{align*}\P_x(\Lambda_z(X_{u_1}) > \epsilon | \tau_z > u_1) &\geq \P_x(\Lambda_z(X_{u_1}) > \epsilon, \tau_z > u_1) \\ &\geq \P_x(\tau_{\Lambda_z^{-1}(\epsilon)} > u_1) \\&\geq \P_{\Lambda_z^{-1}(1/(2A'-1))}(\tau_{\Lambda_z^{-1}(\epsilon)} > u_1) \\& \geq \P_{\Lambda_{h_{max}}^{-1}(1/(2A'-1))}(\tau_{\Lambda_{h_{max}}^{-1}(\epsilon)} > u_1) > 0.\end{align*}
Finally, there exist  $\epsilon \in (0,1/(2A'-1))$ and $c > 0$ (not depending on $z$) such that, for any $x \geq z$, 
    $$\P_x(\Lambda_z(X_{u_1}) \geq \epsilon | \tau_z > u_1) \geq c.$$
\medskip\noindent\textit{Step 2.  Mimicking the steps 2 and 3 in the proof of \cite[Theorem~4.1]{CV2017b}}.

 Now, taking the exact same reasoning as the one presented in the second step of the proof of Theorem 4.1 \cite[Subsection\,5.1]{CV2017b}, one can prove that, for any $z \in [0,h_{max}]$, for all $x \geq \epsilon$,
$$\P_{\Lambda_z^{-1}(\epsilon)}(\Lambda_z(X_{u_{2,z}}) \in \cdot, \tau_z > u_{2,z}) \geq c_{1,z} \psi_z,$$
where 
\begin{itemize}
\item $u_{2,z}$ can be any time satisfying 
$ c'_{1,z} := \inf_{y > z} \P_y(\tau_z < u_{2,z}) > 0,$
\item $c_{1,z} := c_{1,z}' \P_{\Lambda_z^{-1}(\epsilon)}(\tau_z > u_{2,z}),$
\item and $\tilde{\nu}_z := \P_{\Lambda_z^{-1}(\epsilon)}(\Lambda_z(X_{u_{2,z}}) \in \cdot | \tau_z > u_{2,z}).$
\end{itemize}
In particular, for $z=0$, one choose $u_{2,0}$ such that
\begin{equation}
\label{t20}
\inf_{y > 0} \P_y(\tau_0 < u_{2,0}) > 0.\end{equation}
Hence, for any $z \in [0,h_{max}]$ and $x > z$,
$$\P_x(\tau_z < u_{2,0}) \geq \P_x(\tau_0 < u_{2,0}) \geq \inf_{y > 0} \P_y(\tau_0 < u_{2,0}) = c'_{1,0}.$$
Hence, for any $z \in [0,h_{max}]$, 
$$c'_{1,z} = \inf_{x > z} \P_x(\tau_z < u_{2,0}) > c'_{1,0}.$$
In other words, we can set for any $z \in [0,h_{max}]$
$$u_{2,z} = u_{2,0}.$$
Hence, one can define for any $z \in [0,h_{max}]$,
$$c_{1,z} := c'_{1,z}  \P_{\Lambda_z^{-1}(\epsilon)}(\tau_z > u_{2,0}),~~~~\tilde{\nu}_z := \P_{\Lambda_z^{-1}(\epsilon)}(\Lambda_z(X_{u_{2,0}}) \in \cdot | \tau_z > u_{2,0}).$$
As a result, doing the same computation as those presented in Step 3 of the proof of Theorem 4.1 in \cite{CV2017b}, and defining $u_0 := u_1 + u_{2,0}$, for any $x > z$, 
$$\P_x(\Lambda_z(X_{u_0}) \in \cdot | \tau_z > u_0) \geq c_{1,z}c \tilde{\nu}_z \geq c c'_{1,0} \P_{\Lambda_{h_{max}}^{-1}(\epsilon)}(\tau_{h_{max}} > u_{2,0}) \tilde{\nu}_z.$$
In conclusion, to get (almost) \eqref{first-cond}, one has to set $\kappa := c c'_{1,0} \P_{\Lambda_{h_{max}}^{-1}(\epsilon)}(\tau_{h_{max}} > u_{2,0})$ and $\psi_z := \P_{\Lambda_z^{-1}(\epsilon)}(X_{u_{2,0}} \in \cdot | \tau_z > u_{2,0})$ and one has 
\begin{equation}
\label{xi}
\P_x(X_{u_0} \in \cdot | \tau_z > u_0) \geq \kappa \psi_z,~~~~\forall x > z.\end{equation}
To get \eqref{first-cond} exactly, just note that the Lemma \ref{petit} (seen in the proof of Theorem \ref{expo-conv-thm}) can be applied to the conditional probabily $\P_x(X_{u} \in \cdot | \tau_z > u)$, in such a way that the inequality \eqref{xi} holds for any $u$ greater than $u_0$. 
\end{proof}

\subsection{Periodic absorbing function}
Before showing that the Assumption (A') is satisfied when $h$ is periodic or converging, we will need to give some hypothesis on the function $V$ as defined in \eqref{sde}. In the both case we will deal with, the absorbing function $h$ will be Lipschitz, i.e. 
$$L := \sup_{s \leq t} \frac{|h(t)-h(s)|}{|t-s|} < \infty.$$
Now we state the assumption we need on the function $V$
\begin{assumption}[Hypothesis on $V$] 
\label{v}
\begin{itemize}
\item $V$ is such that the process $X$ satisfying \eqref{sde} comes down from infinity.
\item $V$ is positive and increasing on $[-Lu_0,\infty)$ (where $u_0$ is mentioned in Lemma \ref{essentiel}).
\item $\sup_{x \in \R} V'(x) - V^2(x) < \infty$.
\end{itemize}
\end{assumption}   
Note that the functions $V : x \to (x-c)^\alpha$ with $\alpha > 1$ and $c > 0$ are suitable functions. 

Now the following proposition is stated and proved :
\begin{proposition}
\label{periodic}
Let $(X_t)_{t \geq 0}$ be a diffusion process following \eqref{sde}, such that Assumption \ref{v} is satisfied. Assume moreover that $h$ is a periodic function, with period $\gamma > 0$.

Then Assumption (A') holds. In particular, there exists a probability measure $\beta_\gamma$ such that, for any $x > h(0)$,
$$\frac{1}{t} \int_0^t \P_x(X_s \in \cdot | \tau_A > t)ds \underset{t \to \infty}{\longrightarrow} \beta_\gamma.$$
\end{proposition}
\begin{proof}[Proof of Proposition \ref{periodic}]
We will show that the two points in Assumption (A') are satisfied.
\begin{enumerate} 
\item Denote by ${\cal T}_{max}$ the set defined by
$${\cal T}_{max} = \{t \geq 0 : h(t) = h_{max}\}.$$
where we recall that $h_{max} = \sup_{s \geq 0} h(s)$.
The main part of this proof is to show that  there exists $C_{max} > 0$ such that, for any $s \in {\cal T}_{max}$ and any $u \in [u_0,u_0+\gamma]$
\begin{equation}
\label{major-part}
\P_{s,x}(X_{s+u} \in \cdot | \tau_{h } > s+u) \geq C_{max} \psi_{h_{max}},~~~~\forall x > h_{max}.  
\end{equation}
where $u_0$ and $\psi_{h_{max}}$ are defined in Lemma \ref{essentiel}.
Then we will generalize \eqref{major-part} to any $s \geq 0$ using Markov property. \\\\
\textit{First step : Proof of \eqref{major-part}}\\
Let $s \in {\cal T}_{max}$.
For any $x > h_{max}$, for any $t \geq 0$, 
\begin{align*}
\P_{s,x}(X_{t+s} \in \cdot | \tau_{h } > s+t) &\geq \frac{\P_x(\tau_{h_{max}} > t)}{\P_{s,x}(\tau_{h } >s+ t)} \P_x(X_{t} \in \cdot | \tau_{h_{max}} > t).
\end{align*}
Using the Champagnat-Villemonais type condition \eqref{first-cond} for $z = h_{max}$, for any $u \geq u_0$,
$$\P_x(X_{u} \in \cdot | \tau_{h_{max}} >u) \geq \kappa \psi_{h_{max}},~~~~\forall x \in (h_{max},\infty)$$ 
Then we obtain for any $u_0 \leq u \leq u_0 + \gamma$,
\begin{align*}
\P_{s,x}(X_{s+u} \in \cdot | \tau_h >s+u) &\geq \frac{\P_x(\tau_{h_{max}} > u)}{\P_{s,x}(\tau_{h} >s+u)} \kappa \psi_{h_{max}} \\
&\geq \frac{\P_x(\tau_{h_{max}} > u_0+\gamma)}{\P_{s,x}(\tau_{h } >s+u_0)} \kappa \psi_{h_{max}}.
\end{align*}
Recalling that $h$ is Lispchitz and that we defined $L = \sup_{s \leq t} \frac{|h(t)-h(s)|}{|t-s|}$, for any $x \in (h_{max},\infty)$,
$$\frac{\P_x(\tau_{h_{max}} > u_{0}+\gamma)}{\P_{s,x}(\tau_{h } >s+u_{0})} \geq \frac{\P_x(\tau_{h_{max}} >u_{0}+\gamma)}{\P_x(\tau_{u \to h_{max} - Lu} > u_{0})},$$
where 
$$\tau_{u \to h_{max} - Lu} := \inf\{t \geq 0 : X_t = h_{max} - L t\}. $$
To show that 
$$\underset{x \in (h_{max},\infty)}{\inf} \frac{\P_x(\tau_{h_{max}} >u_{0}+\gamma)}{\P_x(\tau_{u \to h_{max} - Lu} > u_{0})} > 0$$
using a continuity argument, it is enough to show that
\begin{equation}
\label{0}
\underset{x \to h_{max}}{\liminf} \frac{\P_x(\tau_{h_{max}} >u_{0}+\gamma)}{\P_x(\tau_{u \to h_{max} - Lu} > u_{0})} > 0
\end{equation}
and
\begin{equation}
\label{inf}
\underset{x \to \infty}{\liminf} \frac{\P_x(\tau_{h_{max}} >u_{0}+\gamma)}{\P_x(\tau_{u \to h_{max} - Lu} >u_{0})} > 0.
\end{equation}
\eqref{inf} is obvious since 
$$\underset{x \to \infty}{\liminf} \frac{\P_x(\tau_{h_{max}} >u_{0}+\gamma)}{\P_x(\tau_{u \to h_{max} - Lu} >u_{0})} \geq \underset{x \to \infty}{\lim} \P_x(\tau_{h_{max}} >u_{0}+\gamma) > 0.$$
Thus let us focus on \eqref{0}. 
Our strategy will be to reduce the study to the case of a Brownian motion. Denote by $(M_t)_{t \geq 0}$ the exponential local martingale defined by, for any $t$,
\begin{align*}
 M_{t} &:= \exp\left(-\int_0^{t} V(W_s)dW_s - \frac{1}{2} \int_0^{t} V^2(W_s)ds\right) \\
&= \exp\left( F(W_0) - F(W_t) + \frac{1}{2} \int_0^t (V'(W_s)-V^2(W_s))ds\right),
\end{align*}
where $F$ is a primitive of $V$ that we choose as a positive function on $[-Lu_0, \infty)$ (it is possible since $F$ is necessarily non-decreasing by the assumptions on $V$).
Under $\P_x$ for $x \in (h_{max},h_{max}+1]$, $W_0 = x$ almost surely. Moreover denote by $\tau^W_{h_{max}}$ and $\tau^W_{u \to h_{max} - Lu}$ the following random times :
$$\tau^W_{h_{max}} := \inf\{t \geq 0 : W_t = h_{max}\},$$
$$\tau^W_{u \to h_{max} - Lu} := \inf\{t \geq 0 : W_t = h_{max} - Lt\}.$$
Thus, since $F$ is non-decreasing, the stopped local martingale $(M_{t \land u_0 \land \tau^W_{u \to h_{max} - Lu}})_{t \geq 0}$ is almost surely bounded by $\exp\left(F(h_{max} + 1) + \frac{u_0}{2} \sup_{y \in \R} V'(y)-V^2(y)\right)$ and is therefore a martingale. Likewise,  the stopped local martingale $(M_{t \land u_0 + \gamma \land \tau^W_{h_{max}}})_{t \geq 0}$ is also a martingale.
 By Girsanov's theorem,
\begin{align*} \frac{\P_x(\tau_{h_{max}} >u_{0}+\gamma)}{\P_x(\tau_{u \to h_{max} - Lu} > u_{0})} &=  \frac{\E_x\left(\1_{\tau^W_{h_{max}} >u_{0}+\gamma}M_{u_0 + \gamma \land \tau^W_{h_{max}}}\right)}{\E_x\left(\1_{\tau^W_{u \to h_{max} - Lu} >u_{0}}M_{u_0 \land \tau^W_{u \to h_{max} - Lu}}\right)} \\
&= \frac{\E_x\left(\1_{\tau^W_{h_{max}} >u_{0}+\gamma}M_{u_0 + \gamma }\right)}{\E_x\left(\1_{\tau^W_{u \to h_{max} - Lu} >u_{0}}M_{u_0}\right)}.\end{align*}
For any $x \in (h_{max},h_{max} + 1]$,
\begin{align*}
&\E_x(\1_{\tau^W_{h_{max}} >u_{0}+\gamma}M_{u_0 + \gamma }) \geq \E_x\left(\1_{\tau^W_{h_{max}} >u_{0}+\gamma}M_{u_0 + \gamma } \1_{\sup_{s \in [0,u_0+\gamma]} W_s \leq h_{max} + 2}\right). \end{align*}
On the event $\{\sup_{s \in [0,u_0+\gamma]} W_s \leq h_{max} + 2\}$, 
$$M_{u_0+\gamma} \geq   \exp\left( - F(h_{max}+2) + \frac{u_0+\gamma}{2}  \inf_{s \in [h_{max},h_{max}+2]} (V'(s)-V^2(s))\right) =: \tilde{M}_{u_0+\gamma}.$$
As a result, \begin{align*}&\E_x(\1_{\tau^W_{h_{max}} >u_{0}+\gamma}M_{u_0 + \gamma }) \geq \tilde{M}_{u_0 + \gamma } \E_x\left(\1_{\tau^W_{h_{max}} >u_{0}+\gamma} \1_{\sup_{s \in [0,u_0+\gamma]} W_s \leq h_{max} + 2}\right)\\
 &\geq \tilde{M}_{u_0 + \gamma } \P_x\left({\tau^W_{h_{max}} >u_{0}+\gamma}\right) \inf_{y \in (h_{max},h_{max}+1]}\P_y\left({\sup_{s \in [0,u_0+\gamma]} W_s \leq h_{max} + 2}\middle| \tau^W_{h_{max}} >u_{0}+\gamma\right).
\end{align*}
Noting that 
$$\lim_{y \to h_{max}} \P_y\left({\sup_{s \in [0,u_0+\gamma]} W_s \leq h_{max} + 2}\middle| \tau^W_{h_{max}} >u_{0}+\gamma\right) = \P\left(\sup_{s \in [0,u_0+\gamma]} W^+_s \leq 2\right) > 0,$$
where $(W^+_t)_{t \geq 0}$ is a Brownian meander (see \cite{DIM77}, Theorem 2.1.), we deduce finally that there exists $c > 0$ such that for any $x \in (h_{max},h_{max}+1]$
$$\E_x(\1_{\tau^W_{h_{max}} >u_{0}+\gamma}M_{u_0 + \gamma }) \geq c  \P_x\left({\tau^W_{h_{max}} >u_{0}+\gamma}\right).$$
On the other side, as we said before, $(M_{t \land u_0 \land \tau^W_{u \to h_{max} - Lu}})_{t \geq 0}$ is almost surely bounded by $\exp\left(F(h_{max} + 1) + \frac{u_0}{2} \sup_{y \in \R} V'(y)-V^2(y)\right)$. Hence there exists $d > 0$ such that, for any $x \in (h_{max},h_{max}+1]$,
$$\E_x\left(\1_{\tau^W_{u \to h_{max} - Lu} >u_{0}}M_{u_0}\right) \leq d \P_x(\tau^W_{u \to h_{max} - Lu} > u_{0}).$$
As a result, for any $(h_{max},h_{max}+1]$,
 $$\frac{\P_x(\tau_{h_{max}} >u_{0}+\gamma)}{\P_x(\tau_{u \to h_{max} - Lu} > u_{0})} \geq \frac{c}{d} \frac{\P_x(\tau^W_{h_{max}} >u_{0}+\gamma)}{\P_x(\tau^W_{u \to h_{max} - Lu} > u_{0})}.$$
For any $x > h_{max}$, denote by $ p^W_{h_{max}}(x,\cdot)$ and $p^W_{u \to h_{max}-Lu}(x,\cdot)$ the density functions of $\tau^W_{h_{max}}$ and $\tau^W_{u \to h_{max} - Lu}$ which are known to be equal to
$$ p^W_{h_{max}}(x,t) = \frac{x-h_{max}}{\sqrt{2\pi t^3}} \exp\left(-\frac{(x-h_{max})^2}{2t}\right)$$
and
$$p^W_{u \to h_{max}-Lu}(x,t) = \frac{x-h_{max}}{\sqrt{2\pi t^3}} \exp\left(-\frac{1}{2t}(x-h_{max} + Lt)^2\right).$$
Then, for any $x \in (h_{max},h_{max}+1]$,
$$\frac{\P_x(\tau^W_{h_{max}} >u_{0}+\gamma)}{\P_x(\tau^W_{u \to h_{max} - Lu} > u_{0})} = \frac{\int_{u_{0}+\gamma}^\infty p^W_{h_{max}}(x,t)dt}{ \int_{u_{0}}^\infty p^W_{u \to h_{max}-Lu}(x,t)dt}.$$
By l'Hôpital's rule,  
\begin{align*}
\lim_{x \to h_{max}}  \frac{\P_x(\tau^W_{h_{max}} >u_{0}+\gamma)}{\P_x(\tau^W_{u \to h_{max} - Lu} > u_{0})} &= \lim_{x \to h_{max}}  \frac{\int_{u_{0}+\gamma}^\infty \partial_x p^W_{h_{max}}(x,t)dt}{ \int_{u_{0}}^\infty \partial_x p^W_{u \to h_{max}-Lu}(x,t)dt} \\
&= \frac{\int_{u_{0}+\gamma}^\infty \partial_x p^W_{h_{max}}(h_{max},t)dt}{ \int_{u_{0}}^\infty \partial_x p^W_{u \to h_{max}-Lu}(h_{max},t)dt} > 0 .\end{align*}
As a result,
$$\liminf_{x \to h_{max}} \frac{\P_x(\tau_{h_{max}} >u_{0}+\gamma)}{\P_x(\tau_{u \to h_{max} - Lu} > u_{0})} \geq \frac{c}{d} \lim_{x \to h_{max}} \frac{\P_x(\tau^W_{h_{max}} >u_{0}+\gamma)}{\P_x(\tau^W_{u \to h_{max} - Lu} > u_{0})} > 0.$$
In conclusion, $\underset{x \in (h_{max},\infty)}{\inf} \frac{\P_x(\tau_{h_{max}} >u_{0}+\gamma)}{\P_x(\tau_{u \to h_{max} - Lu} > u_{0})} > 0$ and \eqref{major-part} holds with $C_{max} =  \kappa \times \underset{x \in (h_{max},\infty)}{\inf} \frac{\P_x(\tau_{h_{max}} >u_{0}+\gamma)}{\P_x(\tau_{u \to h_{max} - Lu} > u_{0})}$. \\\\
\textit{Second step : Generalization and conclusion}\\
Now let $s \geq 0$. Then there exists $s' \geq 0$ such that $s + s' \in {\cal T}_{max}$. As a result we can construct a function $g : \R_+ \to \R_+$ as follows
\begin{equation}
\label{g}g(s) = \inf\{s' \geq 0 : s + s' \in {\cal T}_{max}\}.\end{equation}
In particular, $g(s) = 0$ if $s \in {\cal T}_{max}$. Since $h$ is a continuous function, $s + g(s) \in {\cal T}_{max}$ for any $s \geq 0$. Moreover, since $h$ is $\gamma$-periodic, then for any $s \geq 0$, $g(s) \leq \gamma$. \\
Thus, by the semi-flow property of $(\phi_{t,s})_{s \leq t}$, one has for any $x \in E_s$,
\begin{align*}
\P_{s,x}(X_{s+ u_{0} + \gamma} \in \cdot | \tau_{h} >s+ u_{0} + \gamma) &= \phi_{s+u_0+\gamma, s}(\delta_x) \\
&=  \phi_{s+u_0+\gamma,s+g(s)} \circ \phi_{s+g(s),s} (\delta_x) \\
&= \P_{s+g(s),  \phi_{s+g(s),s} (\delta_x)}(X_{s+u_0+\gamma} \in \cdot | \tau_h > s+ u_0 + \gamma).
\end{align*}
 Now by \eqref{major-part}, for any $x > h(s)$,
\begin{align*}
\P_{s+g(s),\phi_{s+g(s),s}(\delta_x)}(X_{ s + u_{0} + \gamma} \in \cdot | \tau_{h} > s+u_{0} + \gamma)
&\geq C_{max} \psi_{h_{max}}.
\end{align*}
since $u_0 + \gamma - g(s) \in [u_0,u_0 + \gamma]$. Hence, for any $s \geq 0$ and $x > h(s)$,
$$\P_{s,x}(X_{s+ u_{0} + \gamma} \in \cdot | \tau_{h} >s+ u_{0} + \gamma) \geq  C_{max} \psi_{h_{max}}.$$
As a result the first condition in Assumption (A') holds denoting for any $s \geq 0$,
$$\nu_s = \psi_{h_{max}},$$
\begin{equation}
\label{depsee}
t_0 = \gamma + u_0,
\end{equation}
$$c_{1} = C_{max}.$$
\item For the second condition of Assumption (A'), we will use some part of the proof of \cite[Theorem 4.1]{CV2017b}. First we recall \cite[Lemma 5.1]{CV2017b} :
\begin{lemma}[Lemma 5.1., \cite{CV2017b}]
\label{lemma-denis}
There exists $a > h_{max}$ such that $\psi_{h_{max}}([a,\infty)) > 0$ and, for any $k \in \N$,
$$\P_a(X_{k u_{0} \land \tau_{h_{max}}} \geq a) \geq e^{-\rho k  u_{0}},$$
with $\rho > 0$.
\end{lemma}
So let $a$ as in the previous lemma. It is shown in \cite{CV2017b} that we can choose $b > a$ large enough such that
\begin{equation}
\label{expo-moment-2}
\sup_{x \geq b} \E_x(e^{\rho \tau_b}) < \infty.
\end{equation}
Using Markov property, for any $s \geq 0$, $t \geq 0$, and for any $s_0 = k_0\gamma$ with $k_0 \in \N$,
\begin{align*}
\P_{s,a}(\tau_h > s+t) &\geq \P_{s,a}(\tau_h > s+s_0 \land \tau_{h_{max}} + t) \\
&\geq \P_{s,a}(X_{s+s_0 \land \tau_{h_{max}}} \geq b, \tau_h > s+s_0 \land \tau_{h_{max}} + t) \\
& \geq \P_a(X_{s_0 \land \tau_{h_{max}}} \geq b) \P_{s+s_0,b}(s+s_0+t < \tau_{h }) \\
& \geq \P_a(X_{s_0 \land \tau_{h_{max}}} \geq b) \P_{s,b}(s+t < \tau_{h }).
\end{align*}  
Then, for $s_0 > 0$ fixed, $C := 1/\P_a(X_{s_0 \land \tau_{h_{max}}} \geq b) < \infty$, and for any $s \leq t$,
$$\P_{s,b}(t < \tau_{h }) \leq C \P_{s,a}(t < \tau_{h}).$$
Thanks to Markov property again, for any $u \leq t \in \R_+$
\begin{align}
\label{again}
\P_a(X_{u \land \tau_{h_{max}}}  \geq a)\P_{s+u,a}(s+t < \tau_{h }) &\leq \P_{s,a}(s+t < \tau_{h}).
\end{align}
According to Markov property, for any $u \in \R_+$,
\begin{align}\P_a(X_{u \land \tau_{h_{max}}}  \geq a) &\geq \P_a(X_{\lfloor \frac{u}{ u_{0}} \rfloor  u_{0} \land \tau_{h_{max}}}  \geq a)\P_a(X_{(u - \lfloor \frac{u}{u_0} \rfloor  u_{0}) \land \tau_{h_{max}}}  \geq a) \notag \\
&\geq C' \P_a(X_{\lfloor \frac{u}{ u_{0}} \rfloor  u_{0} \land \tau_{h_{max}}}  \geq a), 
\label{cprime}
\end{align} 
where
$$C' :=  \inf_{v \in [0, u_{0}]} \P_a(X_{v \land \tau_{h_{max}}}  \geq a) > 0$$
since $v \to \P_a(X_{v \land \tau_{h_{max}}}  \geq a)$ is continuous and $\P_a(X_{v \land \tau_{h_{max}}}  \geq a) > 0$ for any $v \in [0, u_{0}]$. 
Gathering all these inequalities and using also Lemma \ref{lemma-denis}, for any $x \geq b$,
\begin{align}
\label{begining}
\P_{s,x}(t+s < \tau_{h}) &\leq \P_x(\tau_b > t) + \int_0^t \P_{s+u,b}(t+s < \tau_{h}) \P_x(\tau_b \in du) \\
&\leq \sup_{x \geq b} \E_x(e^{\rho \tau_b}) e^{-\rho t} + C \int_0^t \P_{s+u,a}(t+s < \tau_{h }) \P_x(\tau_b \in du) \notag \\
&\leq \sup_{x \geq b} \E_x(e^{\rho \tau_b}) e^{-\rho \lfloor t/ u_{0} \rfloor  u_{0}} \notag \\
&~~~~~~~~ + \frac{C}{C'} \P_{s,a}(s+t < \tau_{h }) \int_0^t \frac{1}{\P_a(X_{\lfloor u/ u_{0} \rfloor  u_{0} \land \tau_{h_{max}}}  \geq a)} \P_x(\tau_b \in du) \notag \\
&\leq \sup_{x \geq b} \E_x(e^{\rho \tau_b}) e^{\rho u_0} e^{- \rho\left(\left\lfloor \frac{t}{u_0} \right\rfloor + 1\right)u_0} +\frac{C}{C'} \P_{s,a}(t+s < \tau_{h }) \int_0^t e^{\rho u} \P_x(\tau_b \in du) \notag \\
&\leq \sup_{x \geq b} \E_x(e^{\rho \tau_b}) e^{\rho u_0} \P_a(X_{\left(\left\lfloor t/u_0 \right\rfloor + 1\right)u_0 \land \tau_{h_{max}}} \geq a) \notag \\ &~~~~~~~~~~ +\frac{C}{C'} \P_{s,a}(t+s < \tau_{h}) \int_0^t e^{\rho u} \P_x(\tau_b \in du) \notag \\
& \leq \sup_{x \geq b} \E_x(e^{\rho \tau_b}) e^{\rho u_0} \P_{s,a}(\tau_{h } > s+t) +\frac{C}{C'} \P_{s,a}(t+s < \tau_{h}) \int_0^t e^{\rho u} \P_x(\tau_b \in du).
\label{end}
\end{align}
We deduce from \eqref{expo-moment-2} that, for any $t \geq 0$,
$$\sup_{x \geq b} \P_{s,x}(t+s < \tau_{h }) \leq C'' \P_{s,a}(t +s< \tau_{h }),$$
where
$$C'' = \left(e^{\rho u_0} + \frac{C}{C'}\right) \sup_{x \geq b} \E_x(e^{\rho \tau_b}) < \infty.$$
Since $\psi_{h_{max}}([a,\infty)) > 0$, we conclude the point 2. of Assumption (A') setting
$$c_{2} = \frac{1}{C''}.$$
\end{enumerate}
\end{proof}

\subsection{When $h$ is decreasing and converges at infinity}

Let us now state the main proposition of this subsection : 
\begin{proposition}
\label{c}
Let $(X_t)_{t \geq 0}$ be a diffusion process following \eqref{sde}, such that Assumption \ref{v} is satisfied. Assume moreover that $h$ is a decreasing $\cC^1$-function going to $0$ as $t$ goes to infinity.

Then Assumption (A') holds.
\end{proposition}

Since this is a diffusion process on $\R_+$, $(X_t)_{t \geq 0}$ satisfies the strong Markov property and the assumption of continuity presented in Assumption $(H_{hom})$. Moreover, since $t \to X_t$ is continuous almost surely and, for any $s \geq 0$, $\tau_{h(s)}$ is the hitting time of the closed set $[-1,h(s)]$, then $\tau_{h(s)} \underset{s \to \infty}{\longrightarrow} \tau_0$ almost surely, which entails the convergence in law of the hitting times of Assumption $(H_{hom})$. In other words, Assumption $(H_{hom})$ is satisfied for such a process.

Moreover, by \cite[Theorem 4.1.]{CV2017b}, $(H_\infty)$ is satisfied and there exists a unique quasi-stationary distribution $\alpha_\infty \in \cM_1((0,+\infty))$ and two constants $C_\infty, \gamma_\infty$ such that, for any $t \geq 0$ and initial measure $\mu$,
$$\| \P_\mu(X_t \in \cdot | \tau_0 > t) - \alpha_\infty \|_{TV} \leq C_\infty e^{- \gamma_\infty t},$$
as well as a function $\eta_\infty$ as defined in \eqref{def-eta} in the section \ref{general-results}. However, the assumption $(H'_\infty)$ is not satisfied for all decreasing $\cC^1$-function converging to $0$. Nevertheless, it will be satisfied if 
$$h : t \mapsto e^{-\lambda t}$$
with $\lambda > 0$. As a matter of fact, for such a function $h$, one has, by continuity of $(X_t)_{t \geq 0}$ and $h$, for any $s \geq 0$ and $x \in E_s$,
\begin{align*}
\E_{s,x}(e^{\lambda_\infty \tau_h}\eta_\infty(X_{\tau_h})) &= \E_{s,x}(e^{\lambda_\infty \tau_h}\eta_\infty(h(\tau_h))) \\
&= \E_{s,x}(e^{\lambda_\infty \tau_h}\eta_\infty(e^{-\lambda \tau_h})), 
\end{align*} and, using \cite[Proposition 4.2.]{CV2017b}, there exists $K > 0$ such that, for any $x \in (0, + \infty)$, 
$$\eta_\infty(x) \leq K x,$$
so that
$$\E_{s,x}(e^{\lambda_\infty \tau_h}\eta_\infty(X_{\tau_h}))
= \E_{s,x}(e^{\lambda_\infty \tau_h}\eta_\infty(e^{-\lambda \tau_h})) \leq K \E_{s,x}(e^{(\lambda_\infty - \lambda) \tau_h}).$$
Now it is well-known (see \cite[Proposition 3]{MV2012}) that, since $\lambda_\infty - \lambda < \lambda_\infty$, there exists $x_0 \in (0, + \infty)$ such that
$$\E_{x_0}(e^{(\lambda_\infty - \lambda) \tau_0}) < + \infty.$$ 
Hence, for any $s \geq 0$ such that $h(s) \leq x_0$,
\begin{align*}
\E_{s,x_0}(e^{\lambda_\infty  \tau_h}\eta_\infty(X_{\tau_h})) &\leq  K \E_{s,x_0}(e^{(\lambda_\infty - \lambda) \tau_h}) \\
&\leq  K \E_{x_0}(e^{(\lambda_\infty - \lambda) (s + \tau_0)}) < \infty.
\end{align*}
Moreover, 
$$\E_{s,x_0}(e^{\lambda_\infty(\tau_h - s)} \eta_\infty(X_{\tau_h})) \leq K e^{-\lambda s} \E_{x_0}(e^{(\lambda_\infty - \lambda) \tau_0}) \underset{s \to \infty}{\longrightarrow} 0,$$
so the condition $(H'_\infty)$ is satisfied. Then, for such a function $h$, Proposition \ref{c} entails that, for any $\mu \in \cM_1((h(0),+\infty))$,
$$\P_\mu(X_t \in \cdot | \tau_h > t) \underset{t \to \infty}{\overset{(d)}{\longrightarrow}} \alpha_\infty.$$

\begin{proof}[Proof of Proposition \ref{c}]
\begin{enumerate} 
\item Adapting exactly the same reasoning as Proposition \ref{periodic}, we can show that for any $s \geq 0$ and any $x > h(s)$,
$$\P_{s,x}(X_{s+u_{0}} \in \cdot | \tau_{h} >s+ u_{0}) \geq \tilde{d}_s \kappa_{0} \psi_{h(s)},$$
where we recall that $u_0, \kappa_{0}$ and $\psi_z$ are such that \eqref{first-cond} holds, and where $d_s$ is defined by
$$\tilde{d}_s =  \frac{\P_x(\tau_{h(s)} >u_{0})}{\P_x(\tau_{u \to h(s) - Lu} > u_{0})}.$$ 
We have therefore to show that
\begin{equation*}
\inf_{s \geq 0} \tilde{d}_s > 0.
\end{equation*}
For any $z \in [0,h(0)]$ define $(X_t^{(z)})_{t \geq 0}$ by the solution of 
$$dX^{(z)}_t = dW_t - V(X^{(z)}_t +z)dt.$$
In particular, $X^{(0)} \overset{(d)}{=} X$. Likewise, for any $y \in \R$ and $z \in [0,h(0)]$, we denote by $\tau_y^{(z)} := \inf\{t \geq 0 : X^{(z)}_t = y\}$ and $\tau_{u \to y - Lu}^{(z)} := \inf\{t \geq 0 : X^{(z)}_t = y - Lt\}$.
Since $V$ is positive and increasing on $[-Lu_0,\infty)$, then, using Theorem 1.1 in [\cite{IW1989}, Chapter VI, p.437], we can show that for any $x > 0$ and $z \in [0,h(0)]$,
$$\P_x(\tau^{(z)}_0 > u_0) \geq \P_x(\tau^{(h(0))}_0 > u_0)$$
and that
$$\P_x(\tau^{(z)}_{u \to -Lu} > u_0) \leq \P_x(\tau^{(0)}_{u \to -Lu} > u_0).$$
Then, for any $x > 0$ and $s \geq 0$,
\begin{align*}
\P_{x+h(s)}(\tau_{h(s)} > u_0) &= \P_{x}(\tau^{(h(s))}_{0} > u_0) \\
 &\geq  \P_x(\tau^{(h(0))}_0 > u_0) \\
&\geq \frac{ \P_x(\tau^{(h(0))}_0 > u_0)}{\P_x(\tau^{(0)}_{u \to -Lu} > u_0)}  \P_{x+h(s)}(\tau_{u \to h(s)-Lu} > u_0).
\end{align*}
To conclude, it is enough to see that $\inf_{x > 0} \frac{ \P_x(\tau^{(h(0))}_0 > u_0)}{\P_x(\tau^{(0)}_{u \to -Lu} > u_0)} >  0$ using the same techniques as the point $1$ of Proposition \ref{periodic}. \\
As a result the first hypothesis of Assumption (A') holds setting for any $s \geq 0$,
$$\nu_s = \left\{
    \begin{array}{ll}
        \psi_{h(0)} & \mbox{if } s \leq u_0 \\
        \psi_{h(s-u_0)} & \mbox{if } s > u_0
    \end{array}
\right.
$$
$$t_0 = u_0,$$
$$c_{1} = \kappa_0 \times \inf_{s \geq 0} \tilde{d}_s. $$     


\item
The reasoning is the same as the point 2. in the proof of the Proposition \ref{periodic} and the technical computations could be hidden if they are already explicitly written for the periodic case.

Noting that,
for any $z \in [0,h(0)]$ and any $y \geq h(0)$, $\psi_z([y,\infty)) > 0$, then, by Lemma \ref{lemma-denis}, there exists $a > h(0)$ such that, for any $z \in [0,h(0)]$, $\psi_{z}([a,\infty)) > 0$ and for any $k \in \N$
$$\P_{a}(X_{ku_{0} \land \tau_{h(0)}} \geq a) \geq e^{-\rho k u_{0}},$$
where $\rho > 0$. We deduce that for any $s \geq 0$
$$\P_{a}(X_{ku_{0} \land \tau_{h(s)}} \geq a) \geq e^{-\rho k u_{0}}.$$
As in the proof of Proposition \ref{periodic}, we can choose $b > a$ large enough such that
\begin{equation*}
\label{expo-moment-3}
\sup_{x \geq b} \E_x(e^{\rho \tau_{b}}) < \infty.
\end{equation*}
Since $h$ is non-increasing, for any $s,t \geq 0$ and $s_0 \geq 0$,
$$\P_{s+s_0,b}(\tau_{h} > s+s_0+t) \geq \P_{s,b}(\tau_{h } >s+ t).$$
Hence, according to Markov property,
\begin{align*}
\P_{s,a}(s+t < \tau_{h }) &\geq \P_{a}(X_{s_0 \land \tau_{h(s)}} \geq b) \P_{s,b}(s+t < \tau_{h}) \\
&\geq  \P_{a}(X_{s_0 \land \tau_{h(0)}} \geq b) \P_{s,b}(s+t < \tau_{h }). 
\end{align*}
for any $t \geq 0$ and any $s_0 \geq 0$. Hence, for $s_0$ fixed,  $C := \frac{1}{\P_a(X_{s_0 \land \tau_{h(0)}} \geq b)} < \infty$, and for any $s \leq t$,
$$\P_{s,b}(t < \tau_{h }) \leq C \P_{s,a}(t < \tau_{h }).$$
Likewise, one finds an analog of the inequality \eqref{again} 
\begin{align*}
\P_{a}(X_{u \land \tau_{h(s)}} \geq a)\P_{s+u,a}(t+s < \tau_{h }) &\leq \P_{s,a}(t < \tau_{h }),
\end{align*}
and using the same reasoning as for the inequality \eqref{cprime},
\begin{align*}\P_{a}(X_{u \land \tau_{h(s)}}  \geq a) 
&\geq C' \P_{a}(X_{\lfloor \frac{u}{u_{0}} \rfloor u_{0} \land \tau_{h(s)}}  \geq a), 
\end{align*} 
with
$$C' :=  \inf_{v \in [0,u_{0}]} \P_{a}(X_{v \land \tau_{h(0)}}  \geq a) > 0.$$
Hence, using these previous inequalities and doing again the array of computation \eqref{begining}-\eqref{end},
we deduce that, for any $s \leq t$,
$$\sup_{x \geq b} \P_{s,x}(t < \tau_{h }) \leq C'' \P_{s,a}(t < \tau_{h }),$$
where
$$C'' := \left(e^{\rho u_0} + \frac{C}{C'}\right) \sup_{x \geq b} \E_x(e^{\rho \tau_{b}}) < \infty.$$
Since $\psi_{h(s)}([a,\infty)) > 0$ for any $s \geq 0$, we conclude the proof of the point 2 setting
$$c_{2} = \frac{1}{C''}.$$

\end{enumerate}

\end{proof}



\textbf{Acknowledgement} I am very grateful to my advisor Patrick Cattiaux for his relevant remarks and for the attention he gave to this paper, as well as the anonymous referee for its comments.

\bibliographystyle{abbrv}
\bibliography{biblio-william}

\end{document}